\documentclass[a4paper,11pt]{amsart}
\usepackage[utf8]{inputenc}
\usepackage[english]{babel}
\usepackage{graphicx}
\usepackage{amsmath,amssymb, amsbsy, amsfonts, amsthm}
\usepackage[text={16cm,21.5cm},centering]{geometry}
\usepackage{mathrsfs}
\usepackage{caption}
\usepackage{subcaption}
\usepackage{latexsym}
\usepackage{tcolorbox}
\usepackage{enumitem}
\usepackage{float}
\usepackage[colorlinks=true, pdfstartview=FitV, linkcolor=blue, citecolor=blue, urlcolor=blue,pagebackref=false]{hyperref}
\usepackage{url}
\usepackage[toc]{appendix}
\usepackage{xcolor}
\usepackage{color,psfrag}
\usepackage{etoc}
\usepackage{minitoc}
\usepackage{comment}

\newtheorem{theorem}{Theorem}[section]
\newtheorem{corollary}{Corollary}[section]
\newtheorem{lemma}{Lemma}[section]
\newtheorem{definition}{Definition}[section]
\newtheorem{proposition}{Proposition}[section]

\newtheorem{remark}{Remark}[section]

\numberwithin{equation}{section}
\numberwithin{figure}{section}

\newcommand{\R}{\mathbb{R}}
\def\U{{\mathcal U}}
\def\P{{\mathcal P}}
\newcommand{\eps}{\varepsilon}

\newcommand{\norm}[1]{\left \lVert #1 \right \rVert}

\makeatletter
\renewcommand*\env@matrix[1][*\c@MaxMatrixCols c]{
	\hskip -\arraycolsep
	\let\@ifnextchar\new@ifnextchar
	\array{#1}}
\makeatother

\def\eq#1{(\ref{#1})}
\def\neweq#1{\begin{equation}\label{#1}}
	\def\endeq{\end{equation}}

\begin{document}
	
	\title[On spiral steady flows for the Couette-Taylor problem]{On spiral steady flows for the Couette-Taylor problem}
	\author{Edoardo Bocchi, Filippo Gazzola and Antonio Hidalgo-Torné}
	\address[Edoardo Bocchi]{Dipartimento di Matematica, Politecnico di Milano, Piazza Leonardo da
		Vinci 32, 20133 Milano, Italy}
	\email{edoardo.bocchi@polimi.it}
    \address[Filippo Gazzola]{Dipartimento di Matematica, Politecnico di Milano, Piazza Leonardo da
		Vinci 32, 20133 Milano, Italy}
        \email{filippo.gazzola@polimi.it}
	\address[Antonio Hidalgo-Torné]{Max Planck Institute for Mathematics in the Sciences, Inselstrasse 22, 04103 Leipzig, Germany}
	\email{antonio.hidalgo@mis.mpg.de}
	\date{}
	\vspace*{-6mm}
	
	\begin{abstract}

We investigate the Couette-Taylor problem for a steady incompressible viscous fluid in a 3D cylindrical annulus, where one of the two cylinders is still, under both Dirichlet and boundary conditions involving the vorticity that naturally appear in the weak formulation.
The outcome of this study is twofold. 
First, we explicitly determine all the solutions with a specific geometric \emph{partial invariance},
which coincide with the so-called spiral Poiseuille or Poiseuille-Couette flows depending on the boundary conditions. Second, for small boundary data, we provide stability of such solutions, that is, no steady finite-energy perturbations are admissible. To achieve this result in presence of vorticity boundary conditions, we find a substantial analytical difference depending on which cylinder is still.
		 \bigskip
   
        \noindent
		{\bf Mathematics Subject Classification:} 35Q30, 35B06, 76D03, 35B53, 34B24.\par\noindent
		{\bf Keywords:} steady Navier-Stokes equations, cylindrical annulus, invariant solutions, stability.
	\end{abstract}
    
	\maketitle
	\section{Introduction}
	 A long-standing question about the steady motion of an incompressible homogeneous viscous fluid in a cylindrical annulus is the description of all the solutions to the so-called Couette-Taylor problem. Given $0<R_1<R_2$, we consider the 3D unbounded domain
	$$
	\begin{array}{c}
		\Omega=\left\{(x,y,z)\in \R^3:\, R_1^2<x^2+y^2<R_2^2,\,\  z\in\R\right\}\, 
	\end{array}
	$$
    with boundary $\partial\Omega=\Gamma_1\cup\Gamma_2$, where
$$
	\begin{array}{c}
		\Gamma_1=\left\{(x,y,z)\in \R^3:\, x^2+y^2=R_1^2,\, z\in\R\right\} ,\quad \Gamma_2=\left\{(x,y,z)\in \R^3:\, x^2+y^2=R_2^2,\, z\in\R\right\}\,  .
	\end{array}
	$$
  The steady motion of the fluid is governed by the incompressible Navier-Stokes equations
	\begin{equation}\label{SNSE}
		-\Delta u+u\cdot\nabla u+\nabla p=0 \, ,\ \quad  \nabla\cdot u=0 \ \ \mbox{ in } \ \ \Omega\, ,
	\end{equation}
	where one of the two cylinders is still while the other rotates, see Figure \ref{dom1}.	
	\begin{figure}[b]\begin{center}
			\includegraphics[scale=0.25]{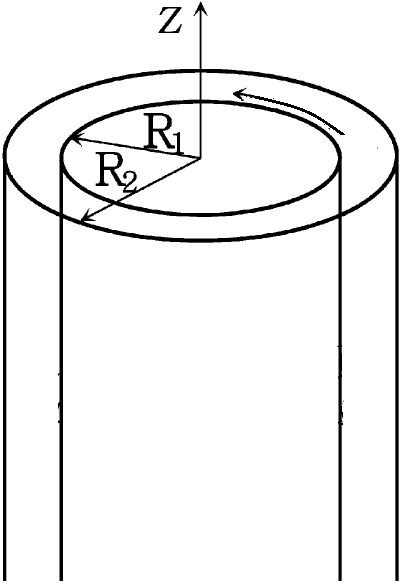}\qquad\qquad\includegraphics[scale=0.25]{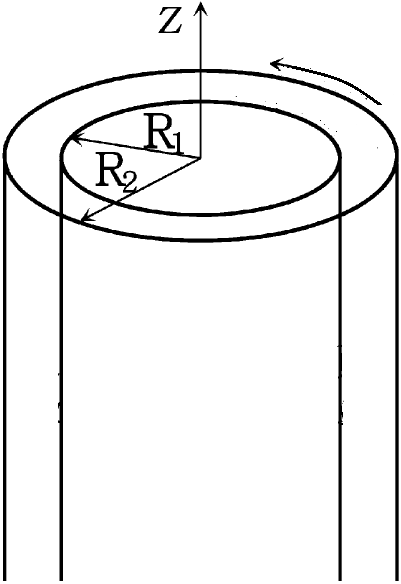}
		\end{center}
		\caption{The cylindrical annulus $\Omega$, with one rotating cylinder.}\label{dom1}
	\end{figure}
	In \eqref{SNSE}, $u:\Omega\to\mathbb{R}^3$ is the velocity vector field, $p:\Omega\to\mathbb{R}$ is the scalar pressure, and the kinematic viscosity is set to unit. As there is no source term in \eqref{SNSE}, the fluid is driven by the movement of one cylinder, which translates into some boundary conditions, see Section \ref{sec-functional}.\par
    {\bf Historical survey.} This is one of the oldest problems in the theory of Navier-Stokes equations.
    So, in order to explain our contribution, let us first briefly retrace some historical facts.
	Since 1890, thanks to Couette \cite{couette} it is known that, for slow rotations of the cylinders, the flow is rotationally-invariant and purely azimuthal.
	In 1917, Rayleigh \cite{rayleigh1917dynamics} studied the dynamics of revolving fluids in annular cylinders.
		In 1921, Taylor \cite{taylor1921experiments} constructed an apparatus where the two cylinders could rotate independently.
		Two years later \cite{taylor}, he investigated the stability of the Couette flow and claimed that the no-slip
		condition was the correct boundary condition for viscous flows at a solid boundary.
		He also experimentally observed that when the rotational velocity of the moving cylinder is increased above a certain threshold,
		the Couette flow becomes unstable and a secondary steady state characterized by axially symmetric toroidal vortices emerges. Upon increasing
		further the rotational velocity of the cylinder the system exhibits more instabilities that lead to states with greater spatio-temporal
		complexity, generating turbulence. The overall findings are nowadays called the Couette-Taylor problem \cite{Chossat}.\par
    This problem attracted a great interest of
    engineers in subsequent years, see e.g.\
\cite{alekseenko1999helical,dou2008instability,luccanegro,ludwieg}. In 1907, Orr \cite[p.75]{orr1907stability} wrote that {\it it would seem improbable that any sharp criterion for stability of fluid motion
		will ever be arrived at mathematically}. And, indeed, it was only more than 40 years after the Taylor experiment that the bifurcation problem
	was tackled theoretically.
	In the seminal paper by Velte \cite{velte1966}, combined with (independently obtained) abstract results by Zeidler \cite{zeidler}
	and Rabinowitz \cite{rabinowitz1971,rabinowitz1972}, it was demonstrated that, for large data there exists at least another solution, still axially symmetric but also vertically-periodic;
	see \cite[Theorem 8A]{zeidlerI}-\cite[Theorem 72C]{zeidlerIV} and \cite[Section II.4]{temam2001navier} for slightly different expositions.\par
The multiplicity of solutions goes far beyond the Taylor 
experiment. Maz'ya \cite[Problem 67]{maz2018seventy} points out a classical issue about \eqref{SNSE} in general domains: to show that uniqueness of
the solution fails for large data. Besides the already mentioned work by Velte \cite{velte1966},
multiplicity results were obtained by Yudovich \cite{yudovich1966secondary,yudovich1967example}, see also a generalization to
forced equations in bodies of revolution around an axis \cite[Theorem IX.2.2]{galdi2011introduction}. For {\em planar flows}, multiplicity results are obtained in a strip \cite{golovkin} and in a square under
{\em Navier boundary conditions} \cite{kochuniqueness}.\par
Multiplicity results trigger the analysis of fluid instabilities and turbulence.
In their monograph fully dedicated to the Couette-Taylor problem, Chossat-Ioos \cite[p.2-3]{Chossat} noticed that {\em observers were astonished
to see the rich variety of patterns that occur, for instance, when the rotation rate of the inner cylinder is increased. ... This model problem then appears as an example of a system that progressively approaches turbulence, which is still one of the biggest challenges for scientists and engineers in fluid mechanics}.
The flow visualization and spectral studies of the Couette-Taylor problem performed in \cite{Andereck} reveals a 
surprisingly large variety of
different states, see also \cite[Fig.I.2,\, p.5]{Chossat} and Figure \ref{chossat}.
\begin{figure}[ht]
\begin{center}
\includegraphics[width=10.5cm]{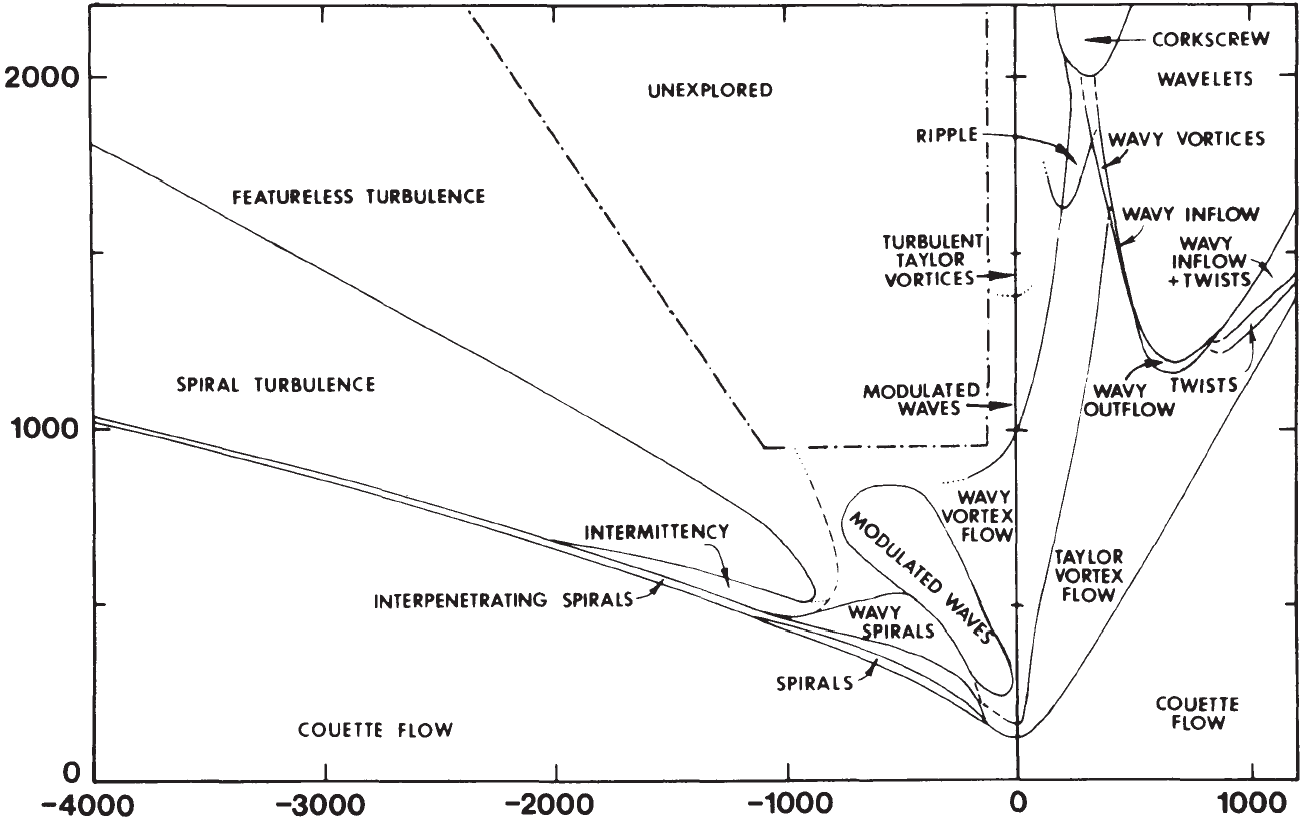}
\caption{Classification of flow states, taken from \cite[Fig.I.2,\, p.5]{Chossat}.}\label{chossat}
\end{center}
\end{figure}
Depending on the angular velocities of the cylinders some ``unexplored regions'' appear. Besides, several of the ``explored regions'' are yet undetectable with theoretical tools. This is related to the following profound consideration \cite{Lokenath} by Hinshelwood (Nobel prize):
{\em fluid dynamicists were divided into hydraulic engineers who observe what cannot be explained and mathematicians who explain things that cannot be observed}.
This is well-illustrated by the story of the Couette-Taylor problem in which many different solutions were physically/numerically
found but only few solutions were mathematically explained. 
A purpose of the present paper is to bring these two 
aspects closer to each other.\par
 {\bf Contributions of the paper.}
Regardless of the magnitude of the data, we explicitly determine all the solutions to \eq{SNSE} with a specific geometric partial invariance, thereby not
entering in the ``unexplored regions'' nor in possible bifurcation patterns. Rayleigh \cite{rayleigh1917dynamics} attempted to classify solutions to \eqref{SNSE} in terms of the dependence of the four
	unknowns in \eqref{SNSE} with respect to the three 
    space variables. We give a rigorous classification in 
    Section \ref{sec-functional}, in line with \cite{Andereck,Chossat,Joseph}.\par
    In Section \ref{Dirichlet} we study \eqref{SNSE} under 
    Dirichlet boundary conditions, by imposing that one cylinder rotates and the other is still.	
For this problem, explicit solutions are known, see \eq{tc1}-\eq{tc2} below, the 
so-called \textit{Couette-Taylor flows} \cite[Chapter II]{landau}. As already mentioned, there may exist other 
solutions depending on the magnitude of the data or, equivalently, on the Reynolds number.
Our first goal is to build a class in which all possible solutions are explicitly determined, regardless of the magnitude of the data.   
This is motivated by an observation by Taylor \cite[p.293]{taylor} who
	wrote: {\it the complexity of the mathematical problem 
    arises from the fact that it is necessary to obtain a 
    three-dimensional solution of
	the equations of motion in which all three 
    components of velocity vanish at both the 
    cylindrical boundaries}. Such solutions do exist and
    are known, the so-called {\em spiral Poiseuille flows} \cite{Joseph,Titi}, usually found by 
    imposing a constant pressure gradient and $z$-periodicity of the velocity.
    These flows are the superposition of the Couette-Taylor flows with
    a vertical Poiseuille flow produced by the pressure gradient along the 
    $z$-axis. Instead, here we find them without these constraints but after introducing suitable {\em partially-invariant solutions}, see Definition
        \ref{kindsolutions}. In Theorems \ref{theo-tcg} 
and \ref{theo-tcg2} we prove that the spiral Poiseuille flows are the only solutions to \eqref{SNSE} in the (fairly 
wide) class of partially-invariant solutions, regardless
of the magnitude of the data. This explicit derivation yields a Liouville-type result for \eqref{SNSE} inside such class of solutions. Within axially symmetric solutions, such result is known without swirl \cite{KochSereginSverak, korobkovpileckas, korobkov-addendum} in the whole $\mathbb{R}^3$ and, under magnitude constraints on the data, in a cylindrical annulus with further restrictions: in the subclass of
vertically-periodic solutions the only admissible is the Couette-Taylor flow (see, e.g.,
\cite[Proposition II.4.2]{temam2001navier}), whereas 
for bounded velocity fields they are the spiral Poiseuille flows, see \cite{kozono2023liouville}. We emphasize that our results 
are obtained without these restrictions and in the class of partially-invariant solutions. By restricting to sufficiently small data,
in Theorems \ref{uniqueness} and \ref{uniqueness2} we prove 
stability results, namely, that {\em no steady finite-energy 
perturbations} of the (infinite-energy) 
partially-invariant solutions are admissible. In Proposition \ref{R1R2vary} we analyze how the smallness condition in Theorem \ref{uniqueness} is affected by the limiting behavior of the radii $R_1$ and $R_2$. The proofs are obtained
through a combination of techniques typical of PDEs and methods from ODEs. Indeed, 
as a byproduct, we obtain a result for a 
nonstandard Sturm-Liouville eigenvalue problem, see
Corollary \ref{coroll}.
\par
The possible failure of uniqueness is the first step towards {\em turbulence}, which is well-measured by the {\em vorticity} of the fluid.  
	In 1827, Navier \cite{navier} proposed boundary conditions with friction, but it was only much later
    that a first contribution appeared in \cite{solonnikov}, followed by many authors whose list is so lengthy that we omit here. It turns out that these slip boundary conditions may be translated in terms of the vorticity, see Section \ref{boundaryvortex}.
    Starting from the work by Pironneau \cite{pironneau}, these conditions, which naturally appear in the weak formulation of \eqref{SNSE}, were rigorously studied in a sound
	functional analytical framework \cite{begue1987nouveau,concafrench,concaenglish}, with applications to fluid flows in a (simply-connected) network of pipes. Motivated by all these facts, in Section 
\ref{boundaryvortex} we study \eqref{SNSE} under boundary conditions involving the vorticity associated with the spiral Poiseuille flows. These boundary conditions also enable us to highlight a substantial analytical difference for the stability depending on which cylinder is still, thus giving theoretical evidence to a phenomenon experimentally observed by Taylor \cite[pp.291-292]{taylor} and clearly depicted in Figure \ref{chossat}.\par
    The so-called \emph{spiral Poiseuille-Couette flows} \cite{Joseph,Titi} are explicit solutions to \eqref{SNSE} found after imposing periodicity of the velocity and a translational motion of the cylinders (sliding). On the contrary, here we find them imposing partial invariance and boundary conditions involving the vorticity. 
    When one cylinder is still, they are explicitly given by \eqref{tcg-vorti} and \eqref{tcg-vorti-outer} below, in which an additional term with respect to the spiral Poiseuille flows appears, showcasing more freedom than with Dirichlet boundary conditions.\par
  In Theorems \ref{theo-tcg-vorti} 
  and \ref{theo-tcg-vorti-outer} we explicitly find all partially-invariant 
  solutions to \eqref{SNSE} under the vorticity boundary conditions on the moving cylinder. 
In order to prove stability for small data, several Poincaré-type inequalities involving the curl of the velocity need to be proved, see 
Lemmas \ref{lem:poincarecurlouter}, \ref{lem:poincarecurlinner}, \ref{lem:poincarecurlperiodic}.
These inequalities are by no means foregone, see Remark
\ref{contro}. With the first curl-Poincaré inequality,
in Theorem \ref{theo:unrestuniquevortouter} we
prove stability of solutions for small data
when the inner cylinder is still. When the outer cylinder is still, the situation is even more involved,
due to the insufficient information provided by the
vorticity on the inner cylinder (concave part of the 
boundary). We manage to prove stability for small data only
under an additional assumption: either for sufficiently ``thin"
cylinders in Theorem \ref{theo-uni-inner} or by considering (infinite-energy) vertically-periodic perturbations in Theorem
\ref{theo:unrestuniquevorinner}, in particular in the 
smaller classes of helical or axially symmetric solutions.

	\section{Boundary conditions and classification of solutions}\label{sec-functional}

	The geometry of $\Omega$ suggests to switch from cartesian coordinates $(x,y,z)\in\R^3$ to cylindrical coordinates
	$(\rho,\theta,z)\in [0,\infty) \times [0,2\pi)\times\R$ and to introduce the orthonormal (planar) basis $\lbrace e_\rho, e_\theta \rbrace \subset \mathbb{R}^{2}$, where
	$$
	e_\rho=(\cos(\theta),\sin(\theta)) \quad \text{and} \quad e_\theta=(-\sin(\theta),\cos(\theta)) \qquad \forall \theta \in [0,2\pi)\, .
	$$
	In the sequel, with an abuse of notation, we denote by $e_\rho$ and $e_\theta$ also the 3D vectors $(e_\rho,0)$ and $(e_\theta,0)$. Then,
	$u(\rho,\theta,z)=u_\rho(\rho,\theta,z)e_\rho+u_\theta(\rho,\theta,z)e_\theta+u_z(\rho,\theta,z)e_z\,
	$
	and \eqref{SNSE} become
	\begin{equation}\label{SNSEcyl}
		\left\{\begin{array}{l}
			\frac{1}{\rho}\frac{\partial u_\rho}{\partial\rho}+\frac{\partial^2 u_\rho}{\partial\rho^2}+\frac{1}{\rho^2}\frac{\partial^2 u_\rho}{\partial\theta^2}+\frac{\partial^2 u_\rho}{\partial z^2}
			-u_\rho\frac{\partial u_\rho}{\partial\rho}-\frac{u_\theta}{\rho}\frac{\partial u_\rho}{\partial\theta}-u_z\frac{\partial u_\rho}{\partial z}
			+\frac{u_\theta^2}{\rho}-\frac{u_\rho}{\rho^2}-\frac{2}{\rho^2}\frac{\partial u_\theta}{\partial\theta}=\frac{\partial p}{\partial\rho}\\[5pt]
			\frac{1}{\rho}\frac{\partial u_\theta}{\partial\rho}+\frac{\partial^2 u_\theta}{\partial\rho^2}
			+\frac{1}{\rho^2}\frac{\partial^2 u_\theta}{\partial\theta^2}+\frac{\partial^2 u_\theta}{\partial z^2}-u_\rho\frac{\partial u_\theta}{\partial\rho}-\frac{u_\theta}{\rho}\frac{\partial u_\theta}{\partial\theta}
			-u_z\frac{\partial u_\theta}{\partial z}-\frac{u_\rho u_\theta}{\rho}-\frac{u_\theta}{\rho^2}+\frac{2}{\rho^2}\frac{\partial u_\rho}{\partial\theta}
			=\frac{1}{\rho}\frac{\partial p}{\partial\theta}\\[5pt]
			\frac{1}{\rho}\frac{\partial u_z}{\partial\rho}+\frac{\partial^2 u_z}{\partial\rho^2}+\frac{1}{\rho^2}\frac{\partial^2 u_z}{\partial\theta^2}+\frac{\partial^2 u_z}{\partial z^2}
			-u_\rho\frac{\partial u_z}{\partial\rho}-\frac{u_\theta}{\rho}\frac{\partial u_z}{\partial\theta}-u_z\frac{\partial u_z}{\partial z}=\frac{\partial p}{\partial z}\\[5pt]
			\frac{\partial u_\rho}{\partial\rho}+\frac{u_\rho}{\rho}+\frac{1}{\rho}\frac{\partial u_\theta}{\partial\theta}
			+\frac{\partial u_z}{\partial z}=0
		\end{array}\right.
	\end{equation}
	for all $(\rho,\theta,z)\in(R_1,R_2)\times[0,2\pi)\times\R$.
	Associated with \eqref{SNSEcyl}, we consider the two different Dirichlet boundary-value problems
	\begin{equation}\label{nsstokes0}
		\begin{aligned}
			& u=\alpha \, e_\theta\quad \text{on}\quad \Gamma_1\, ,\qquad & u =0\quad \text{on}\quad \Gamma_2\, ,\\
			& u=0\quad \text{on}\quad \Gamma_1\, ,\qquad & u =\alpha \, e_\theta\quad \text{on}\quad \Gamma_2\, .
		\end{aligned}
	\end{equation}
The boundary conditions in \eqref{nsstokes0}$_1$ (resp.\ \eqref{nsstokes0}$_2$) show that $\Gamma_1$ (resp.\ $\Gamma_2$) rotates with
constant angular velocity $\alpha/R_1$ (resp.\ $\alpha/R_2$), whereas $\Gamma_2$ (resp.\ $\Gamma_1$) remains still: therefore, the fluid motion
is generated only by the rotation of one cylinder.
The (necessary) compatibility condition $\int_{\Omega}\nabla\cdot u =\int_{\partial\Omega}u \cdot\nu=0$ is satisfied for
solutions to \eq{SNSEcyl}-\eq{nsstokes0}: here, $\nu$ is the outward unit normal to $\Omega$ ($\nu=-e_\rho$ when $\rho=R_1$, $\nu=e_\rho$ when $\rho=R_2$).\par
An explicit solution to \eq{SNSEcyl}-\eq{nsstokes0}$_1$ exists for every $\alpha\in \mathbb{R}$, and is given by $(\alpha{\U^C}e_\theta,\alpha^2{\P^C})$ with
	\begin{equation} \label{tc1}
		{\U^C}(\rho) =\frac{R_1}{R_2^2 -R_1^2} \left(\frac{R_2^2}{\rho}-\rho\right), \
		\P^C(\rho) =\frac{1}{2}\left[\frac{R_1}{R_2^2-R_1^2}\right]^{2} \left( \rho^{2} - \frac{R_2^4}{\rho^{2}}
		- 4R_2^2 \log\rho \right)\, 
	\end{equation}
	for $\rho\in [R_1,R_2]$. Also an explicit solution to \eq{SNSEcyl}-\eqref{nsstokes0}$_2$ exists for every $\alpha\in \mathbb{R}$, and is given by
	$(\alpha{\widetilde\U^C}e_\theta,\alpha^2{\widetilde\P^C})$ with
	\begin{equation} \label{tc2}
		{\widetilde\U^C}(\rho) =\frac{R_2}{R_2^2 -R_1^2} \left(\rho-\frac{R_1^2}{\rho}\right), \
		\widetilde\P^C(\rho) =\frac{1}{2}\left[\frac{R_2}{R_2^2-R_1^2}\right]^{2} \left(\rho^2-\frac{R_1^4}{\rho^{2}}-4R_1^2\log\rho\right) 
	\end{equation}
    for $\rho\in [R_1,R_2]$. We study \eq{SNSEcyl} with boundary conditions involving the vorticity on the moving cylinder derived from \eq{tc1} and \eq{tc2} in Section \ref{boundaryvortex}.\par
	
	Since $\partial \Omega$ and all the boundary data in \eq{nsstokes0} are of class $C^{\infty}$, well-known regularity results
	\cite[Theorem IX.5.2]{galdi2011introduction} imply that any solution to \eqref{SNSEcyl} is smooth and classical, $u\in C^\infty(\overline{\Omega})$,
	and that there exists an associated pressure $p \in C^\infty(\overline{\Omega})$ such that the pair $(u,p)$ solves \eqref{SNSEcyl}
	pointwise. These functions are also $2\pi$-periodic with respect to $\theta$.\par
    By exploiting the axial symmetry of $\Omega$, we now classify the solutions to \eqref{SNSEcyl}.
	In the cartesian coordinates system, consider the rotation matrix about the $z$-axis by an angle $\phi\in[0,2\pi)$
	\begin{equation} \label{rotmat}
		\mathcal{R}_\phi :=
		\begin{bmatrix}
			\cos(\phi) & -\sin(\phi)  & 0\\[0.2cm]
			\sin(\phi) & \cos(\phi)   & 0\\[0.2cm]
			0          &    0         & 1
		\end{bmatrix},
	\end{equation}
	and its combination with vertical translations
	$$
	\begin{bmatrix}
		x\\
		y\\
		z
	\end{bmatrix}
	\ \mapsto \mathcal{S}_{\phi,\gamma}\begin{bmatrix}
		x\\
		y\\
		z
	\end{bmatrix}
	:=\mathcal{R}_\phi\begin{bmatrix}
		x\\
		y\\
		z
	\end{bmatrix}+\gamma\begin{bmatrix}
		0\\
		0\\
		\phi
	\end{bmatrix}
	\qquad\forall\phi\in[0,2\pi)\, ,\ \gamma \in\R\, .
	$$
	Then, we classify solutions according to
	
	\begin{definition}\label{kindsolutions}
		A classical solution $(u,p)$ to \eqref{SNSEcyl} is called:\par
		$\bullet$ partially-invariant if $u_\theta$ and $p$ do not depend on $\theta$, $u_z$ does not depend on $z$;\par
		$\bullet$ $z$-periodic if $u(\rho,\theta,z+T)=u(\rho,\theta,z)$ and $p(\rho,\theta,z+T)=p(\rho,\theta,z)$
		in $\Omega$ for some $T>0$;\par
		$\bullet$ helical if there exists $\gamma\neq0$ such that $\mathcal{R}_\phi u(\xi)=u(\mathcal{S}_{\phi,\gamma}\xi)$ for all $\xi\in\R^3$ and $\phi\in[0,2\pi)$;\par
		$\bullet$ axially symmetric if $u$ and $p$ do not depend on $\theta$.
	\end{definition}
	
	Roughly speaking, partially-invariant solutions are solutions for which $u_\theta$ and $p$ (resp.\ $u_z$) does not depend on
	$\theta$ (resp.\ $z$), while there is no such request on the radial velocity component $u_\rho$, namely,
	\begin{equation}\label{assumption}
		u_\rho=u_\rho(\rho,\theta,z)\, ,\quad u_\theta=u_\theta(\rho,z)\, ,\quad u_z=u_z(\rho,\theta)\, ,\quad p=p(\rho,z).
	\end{equation}
	Simple examples of solutions to \eqref{SNSEcyl} belonging to {\em all} classes in Definition \ref{kindsolutions} are
	$(\alpha\U^C e_\theta,\alpha^2\P^C)$ and $(\alpha{\widetilde\U}^Ce_\theta,\alpha^2{\widetilde\P^C})$, see \eqref{tc1}-\eqref{tc2}.
	
	Periodic axially symmetric solutions were considered in \cite{velte1966}, which proves bifurcations results.
	Helical invariant solutions were found in {\it bounded domains} with helical symmetry in \cite{korobkov2022existence}.

	\section{Dirichlet boundary conditions}\label{Dirichlet}

	\subsection{Explicit derivation of partially-invariant solutions with still outer cylinder.}\label{sec-parinv}

    Under the sole Dirichlet conditions \eqref{nsstokes0}, the boundary value problem for equations \eq{SNSEcyl} is {\em underdetermined} because
	there are no restrictions in the $z$-direction. For this reason, one expects the existence of multiple solutions, in particular
	infinite-energy solutions. To this end, we consider the circular Poiseuille function
    \begin{equation}\label{Ualphabeta}
        \U^P(\rho):=
        \frac{1}{4} \left(\rho^2-R_1^2 - \frac{R_2^2-R_1^2}{\log (R_2/R_1)}\log(\rho/R_1) \right)\qquad \text{for} \quad \rho\in [R_1,R_2].
    \end{equation}

	In our first result we determine all the partially-invariant solutions to \eq{SNSEcyl}-\eqref{nsstokes0}$_1$. We emphasize that, differently from \cite{kozono2023liouville}, such solutions have only $u_\theta$ and $p$ axially symmetric and we do not require boundedness of $u$.
	
	\begin{theorem}\label{theo-tcg}
		Let $({\U^C},{\P^C})$ be as in \eqref{tc1}, let $\U^P$ be as in \eqref{Ualphabeta}. Then, the spiral Poiseuille flows 
		\begin{equation} \label{tcg}\begin{aligned}
				&\U_{\alpha,\beta}(\rho) := \alpha\U^C(\rho)e_\theta+ \beta\U^P(\rho)e_z,\\[5pt]
				& \P_{\alpha,\beta}(\rho, z) := \alpha^2\P^C(\rho)+ \beta z,\end{aligned}  \qquad \text{for} \quad (\rho,z) \in[R_1,R_2]\times \mathbb{R}\,
		\end{equation}
		are the only partially-invariant solutions  to \eqref{SNSEcyl}-\eqref{nsstokes0}$_1$ as $ \beta$ varies in $\mathbb{R}$.
		\end{theorem}
	\begin{proof} Let $u$ be a partially-invariant solution to \eqref{SNSEcyl}. We divide the proof in two steps. \\
		\underline{Step 1:} we show that the following alternative holds:
		\begin{equation}\label{either}
			\mbox{either }u_\theta\mbox{ does not depend on }z\ (u_\theta=u_\theta(\rho))
			\mbox{ or }u_z\mbox{ does not depend on }\theta\ (u_z=u_z(\rho))\, .
		\end{equation}By \eqref{assumption}, the incompressibility condition becomes
		\begin{equation}\label{incom-cylin}
			\frac{\partial}{\partial\rho}\Big(\rho u_\rho(\rho,\theta,z)\Big)=0 \ \Longrightarrow \ u_{\rho}(\rho,\theta,z)
			=\frac{A(\theta,z)}{\rho} \quad  \forall (\rho,\theta,z)\in(R_1,R_2)\times (0, 2\pi)\times \mathbb{R}\, ,
		\end{equation}
		for some smooth $A:\R^2\to\mathbb{R}$, $2\pi$-periodic with respect to $\theta$. The boundary conditions \eq{nsstokes0}$_1$ imply $A(\theta,z)\equiv0$ and $u_{\rho}(\rho, \theta, z)\equiv0$.
		Therefore, the first three equations in \eqref{SNSEcyl} in the class of partially-invariant solutions reduce to
		\begin{eqnarray}
			\qquad \frac{\partial p}{\partial \rho}(\rho, z) &=& \frac{u^{2}_{\theta}(\rho,z)}{\rho}\, ,\label{uno}\\
			0 \quad &=& \rho\frac{\partial^2 u_{\theta}}{\partial\rho^2}(\rho,z) +
			\frac{\partial u_{\theta}}{\partial\rho}(\rho,z)+\rho\frac{\partial^2 u_{\theta}}{\partial z^2}(\rho,z)
			-\rho u_z(\rho,\theta)\, \frac{\partial u_{\theta}}{\partial z}(\rho,z)- \frac{u_\theta(\rho,z)}{\rho}\, ,\label{dos}\\
			\frac{\partial p}{\partial z}(\rho, \theta) &=& \frac{\partial^2 u_z}{\partial\rho^2}(\rho,\theta) +\frac{1}{\rho}\frac{\partial u_z}{\partial\rho}(\rho,\theta)+\frac{1}{\rho^2}\frac{\partial^2 u_z}{\partial\theta^2}(\rho,\theta)-
			\frac{u_\theta(\rho,z)}{\rho}\, \frac{\partial u_z}{\partial\theta}(\rho,\theta)\, .\label{tres}
		\end{eqnarray}
		From \eqref{dos} we obtain that
		\begin{equation}\label{sinB}
			\rho u_z(\rho,\theta)\, \frac{\partial u_{\theta}}{\partial z}(\rho,z)=\rho\frac{\partial^2 u_{\theta}}{\partial\rho^2}(\rho,z) +
			\frac{\partial u_{\theta}}{\partial\rho}(\rho,z)+\rho\frac{\partial^2u_{\theta}}{\partial z^2}(\rho,z)-\frac{u_\theta(\rho,z)}{\rho}
		\end{equation}
		for all $(\rho, \theta, z)\in (R_1,R_2)\times (0, 2\pi)\times \mathbb{R}$. Since the term on the right-hand side of \eqref{sinB} does not depend on $\theta$ neither does the term on the left, which implies that
		$$
		\text{either}\quad \frac{\partial u_{\theta}}{\partial z}(\rho,z)=0,\quad \text{or}\quad u_z(\rho,\theta)=u_z(\rho),
		$$
		which is equivalent to \eqref{either}.\\
		\underline{Step 2:} we prove that, necessarily, $(u,p)=(\U_{\alpha,\beta}, \P_{\alpha,\beta})$ given by \eqref{tcg}
		for some $\beta\in\R$. To this end, we distinguish two cases.\par
		(a) If the first alternative in \eqref{either} holds, then \eqref{sinB} becomes
		$$
		\rho^2\frac{d^2 u_{\theta}}{d\rho^2}(\rho)+\rho\frac{du_{\theta}}{d\rho}(\rho)-u_\theta(\rho)=0\quad\text{in}\quad (R_1,R_2)
		$$
		which, complemented with the boundary conditions \eq{nsstokes0}$_1$, gives
		\begin{equation}\label{utheta}
			u_\theta(\rho)=\frac{\alpha R_1}{R_2^2 -R_1^2} \left(\frac{R_2^2}{\rho}-\rho\right) .
		\end{equation}
		By replacing \eqref{utheta} into \eqref{uno} and integrating with respect to $\rho$, we find a separated structure of the pressure, namely,
		$
		p(\rho, z)=D(\rho)+E(z)
		$
		for some smooth functions $D:[R_1,R_2]\to\R$ and $E:\R\to\R$. Then, \eqref{tres} reads
		\begin{equation}\label{eq-uz}
			\frac{\partial^2 u_z}{\partial\rho^2}(\rho,\theta)+\frac{1}{\rho}\frac{\partial u_z}{\partial\rho}(\rho,\theta)+\frac{1}{\rho^2}\frac{\partial^2 u_z}{\partial\theta^2}(\rho,\theta)-\frac{u_\theta(\rho)}{\rho}\, \frac{\partial u_z}{\partial\theta}(\rho,\theta)=E'(z)\, .
		\end{equation}
		Since the left hand side of \eqref{eq-uz} does not depend on $z$, neither does the right hand side and $E'(z)\equiv\beta$ for some $\beta\in\R$.
		Thus,
		\begin{equation}\label{sepa-pres}
			p(\rho,z)=D(\rho)+\beta z
		\end{equation}and,
		recalling \eqref{utheta}, \eqref{eq-uz} reduces to
		\begin{equation}\label{F}
			\frac{\partial^2 u_z}{\partial\rho^2}(\rho,\theta)+\frac{1}{\rho}\frac{\partial u_z}{\partial\rho}(\rho,\theta)+
			\frac{1}{\rho^2}\frac{\partial^2 u_z}{\partial\theta^2}(\rho,\theta)+
			\frac{\alpha R_1 }{R_2^2 -R_1^2} \left(1-\frac{R_2^2}{\rho^2}\right)\, \frac{\partial u_z}{\partial\theta}(\rho,\theta)=\beta\in\R\, .
		\end{equation}
		For $2\pi$-periodic functions $\theta\mapsto v(\rho, \theta)$, we restrict ourselves to the class
		\begin{equation}\label{V}
			V=\left\{v\in H^1((R_1,R_2)\times [0,2\pi));\ v(R_1,\theta)=v(R_2, \theta)=0  \text{ for $\theta\in[0,2\pi)$}\right\},
		\end{equation}which is a closed subspace of $H^1((R_1,R_2)\times (0,2\pi))$. 	We claim that \eqref{F} admits a unique solution in $V$.
		Indeed, after multiplying \eqref{F} by $-\rho$, its variational formulation reads
		\begin{equation}\label{var-form}
			B(u_z, v)= F(v) \ \text{for any} \ v\in V,
		\end{equation}
		where $F(v)=-\beta\int_{(R_1,R_2)\times (0,2\pi)} \rho vd\rho d\theta$ is a linear functional on $V$ and \begin{equation}\label{B}B(u_z, v)=\int_{(R_1,R_2)\times (0,2\pi)} \left(\rho \frac{\partial u_z}{\partial \rho}\frac{\partial v}{\partial \rho} + \frac{1}{\rho}\frac{\partial u_z}{\partial \theta} \frac{\partial v}{\partial \theta} + w_\alpha \frac{\partial u_z}{\partial \theta}v\right) d\rho d\theta\end{equation} is a bilinear form on $V$ with $w_\alpha(\rho)= \frac{\alpha R_1}{R_2^2-R_1^2}\left( \frac{R_2^2}{\rho} -\rho\right)$. Thanks to the Hölder inequality, we have that $|F(v)|\leq C_1\|v\|_V$ and $|B(u_z, v)|\leq C_2 \|u_z\|_{V}\|v\|_{V}$ for some constants $C_1=C_1(R_1,R_2,\beta)$ and $C_2=C_2(R_1,R_2, \alpha)$. Furthermore, using the periodicity with respect to $\theta$ in \eqref{V}, we obtain that
		\begin{equation*}\begin{aligned}
				B(v, v)&= \int_{(R_1,R_2)\times (0,2\pi)} \left( \rho\left|\frac{\partial v}{\partial \rho}\right|^2 + \frac{1}{\rho}\left|\frac{\partial v}{\partial \theta}\right|^2 +\frac{1 }{2}\frac{\partial (w_\alpha v^2)}{\partial \theta} \right) d\rho d\theta\\
				&= \int_{(R_1,R_2)\times (0,2\pi)} \left( \rho \left|\frac{\partial v}{\partial \rho}\right|^2 + \frac{1}{\rho}\left|\frac{\partial v}{\partial \theta}\right|^2  \right) d\rho d\theta\geq \min (R_1, 1/R_2)\|v\|^2_V\qquad\forall v\in V,
			\end{aligned}
		\end{equation*} that is, $B$ is coercive (or $V$-elliptic). Then, for any $(\alpha, \beta)\in \mathbb{R}^2$, the Lax-Milgram Theorem yields the existence and uniqueness of the solution to \eqref{F} in $V$. We now seek the solution $u_z$ as a function depending only on $\rho$, hence it satisfies
		\begin{equation*}
			\begin{cases}
				\dfrac{d^2 {u_z}}{d\rho^2}(\rho) +\dfrac{1}{\rho}\dfrac{d{u_z}}{d\rho}(\rho) =\beta \quad \text{in} \quad (R_1,R_2),\\[5pt]
				{u}_{z}(R_1)=  {u}_{z}(R_2)=0,
			\end{cases}
		\end{equation*}which can be solved to obtain the explicit expression
		\begin{equation}\label{uz-explicit}
			{u}_{z}(\rho)=  \frac{\beta}{4} \left(\rho^2-R_1^2 - \frac{R_2^2-R_1^2}{\log (R_2/R_1)}   \log (\rho/R_1)\right).
		\end{equation}
		Thus, by uniqueness, we infer that \eqref{uz-explicit} is the unique solution to \eqref{F}, and we remark that it does not depend on the parameter $\alpha$ (rotational velocity of the inner cylinder $\Gamma_1$).	
		Furthermore, expliciting the function $D(\rho)$ in \eqref{sepa-pres} using \eqref{utheta} yields
		\begin{equation*}
			p(\rho,z)=   \frac{1}{2}\left(\frac{\alpha R_1}{R_2^2-R_1^2} \right)^{2} \left( \rho^{2} - \frac{R_2^4}{\rho^{2}}
			- 4R_2^2 \log\rho \right) + \beta z,
		\end{equation*}
		where we have set the additive constant equal to zero.
		
		(b) If the second alternative in \eqref{either} holds, on the one hand \eqref{tres} reduces to
		$$
		\frac{\partial p}{\partial z}(\rho) = \frac{d^2 u_z}{d\rho^2}(\rho)+\frac{1}{\rho}\frac{d u_z}{d\rho}(\rho)\,,
		$$
		which implies that $\tfrac{\partial^2p}{\partial z\partial\rho}$ merely depends on $\rho$. On the other hand, by differentiating \eqref{uno} with respect to $z$, we obtain
		$$
		\frac{\partial u^2_{\theta}}{\partial z}(\rho,z)=\rho\frac{\partial^2p}{\partial z\partial\rho}(\rho)=: C(\rho)
		$$and integrating with respect to $z$ yields
		\begin{equation}
			| u_\theta (\rho,z)|= \sqrt{C(\rho)z + K(\rho)}
		\end{equation}
		with some smooth function $K: [R_1,R_2] \rightarrow \mathbb{R}$. Since $u_\theta$ is smooth and defined for all $z\in\mathbb{R}$, it must
		necessarily be $C(\rho)\equiv 0$ and we find that $u_\theta=u_\theta(\rho)$; this is the first alternative in \eqref{either}
		that has already been discussed in case (a). The theorem is so proved in both cases.\end{proof}
	
			The pressure $\mathcal{P}_{\alpha,\beta}$ in \eqref{tcg} is defined up to an additive
            constant and its gradient is uniquely determined by the velocity. This is why throughout the paper by ``flow" we mean both the velocity $u$ and the velocity-pressure couple $(u,p)$ in \eqref{SNSE}. For physical reasons, one expects the pressure to be decreasing with respect to $z$, 
            hence it is customary to take $\beta\leq 0$. We are so led to analyze with some attention the
            map $\rho\mapsto-\U^P(\rho)$, see \eqref{Ualphabeta}. It vanishes at the 
            endpoints $R_1,R_2$ and is strictly convex in $[R_1,R_2]$; it achieves its maximum at
\begin{equation}\label{R0}
	R_0=\sqrt{\frac{R_2^2-R_1^2}{2\log(R_2/R_1)}}.
\end{equation}
	Note that $R_1<R_0<\frac{R_2+R_1}{2}$, so $R_0$ is on the left of the midpoint of the interval and there is an intermediate cylinder (of radius $R_0$) where 
    the fluid velocity $\U_{0,-1}$ attains its maximum value (in modulus). In Figure \ref{figvortex} we sketch the just described behaviour of the velocity, which also holds for $\U_{0,\beta}$ for any $\beta<0$.\par
	
	\begin{figure}[H]
		\begin{center}
			\includegraphics[scale=0.5]{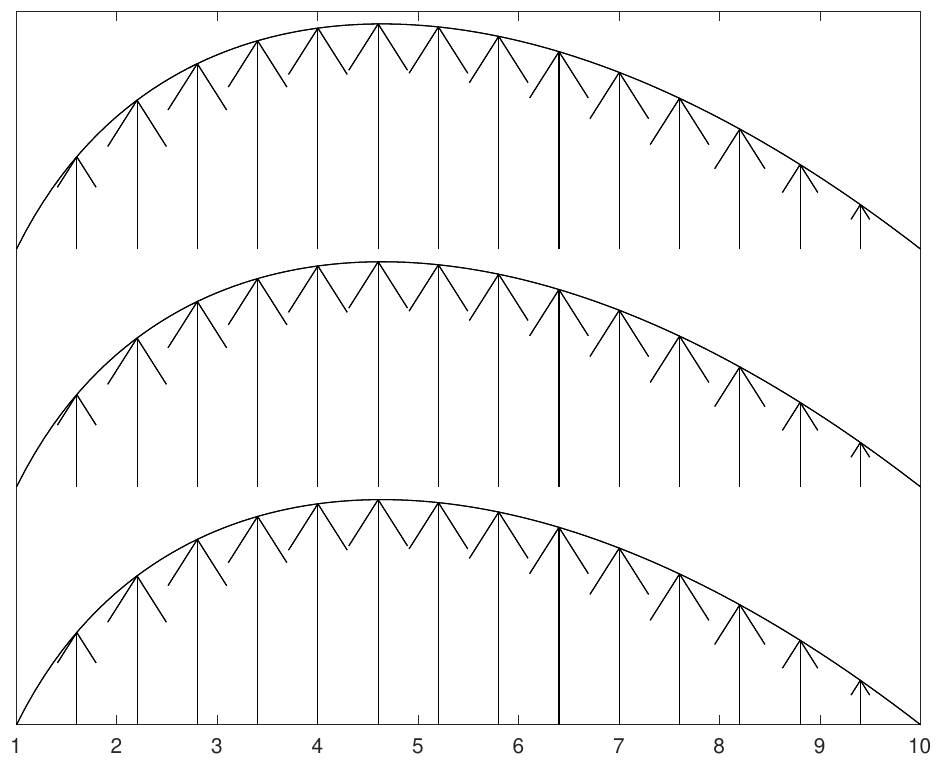}
		\end{center}
		\caption{The Poiseuille flow $\mathcal{U}_{0,-1}$ produced by the pressure gradient along the $z$-axis in the plane $(\rho, z)$ with $R_1=1$ and $R_2=10$ ($R_0\approx 4.64$).}\label{figvortex}
	\end{figure}

Theorem \ref{theo-tcg} also deserves two deeper comments,
one related to its physical consequences and the other related
to the theory of ODEs.\par
First,	we stress that the velocity $\mathcal{U}_{\alpha, \beta}$ of the spiral Poiseuille flows 
 in \eqref{tcg} have infinite 
energy since they do not vanish as $|z|\to\infty$. The 
velocities (but not the pressures!) of
the solutions \eqref{tcg} form a linear vector space $V_{\alpha, \beta}$ although the equations \eqref{SNSEcyl} are nonlinear. The	parameter $\beta$ should
	be seen as the mass density so that $\P_{0,\beta}$ plays the role of the hydrostatic pressure, whereas $\U_{0,\beta}$ has infinite energy. On the other hand, also the circular Couette flow \eq{tc1}
	has infinite energy (with no pressure drop). Moreover, unbounded cylinders do not exist in nature. Hence, the solutions \eqref{tcg} should be seen as (mathematical) limits of finite-energy solutions on
	{\em bounded cylindrical annuli}. The taller are the cylinders, the larger is the pressure drop, increasing linearly with respect to the height of the cylinders.\par
Second, as a byproduct of the proof of Theorem
\ref{theo-tcg}, we obtain a result on the eigenvalues of
the following Sturm-Liouville-type problems
	\begin{equation}\label{SL-pb}
		\begin{cases}
			\dfrac{d}{d\rho}	\left(\rho \dfrac{d W_k}{d\rho}(\rho) \right) - \dfrac{k^2}{\rho} W_k(\rho)= - \lambda\dfrac{k R_1}{R_2^2-R_1^2}\left(\rho-\dfrac{R_2^2}{\rho}\right) W_k(\rho)  \quad \text{in} \quad (R_1,R_2),\\[10pt]
			W_k(R_1)=  W_k(R_2)=0,
		\end{cases}
	\end{equation} where $W_k$ are complex-valued and the weight function multiplying $W_k$ in the right-hand side of \eqref{SL-pb} \emph{vanishes} at the endpoint $\rho=R_2$. This makes  the problem singular and does not enter into the classical (regular) Sturm-Liouville theory. However, we obtain the following statement, which has its own independent interest:
	
	\begin{corollary}\label{coroll}
		Let $R_2>R_1>0$. Then, for any $k\in \mathbb{Z}\setminus \{0\}$, the Sturm-Liouville-type problem \eqref{SL-pb} admits no purely-imaginary eigenvalues $\lambda$.
	\end{corollary}
	\begin{proof}
		When $\lambda=0$, the unique solution to \eqref{SL-pb} is $W_k\equiv0$ and, hence $\lambda=0 $ is not an eigenvalue. For $\lambda\neq 0$, since $\theta \mapsto u_z(\rho, \theta)$ in \eqref{F} is $2\pi$-periodic, we expand it in Fourier series and using the boundary conditions \eq{nsstokes0}$_1$,
		we obtain that the Fourier coefficients $\widehat{u}_{z,k}$ with $k\in \mathbb{Z}\setminus\{0\}$ solve \eqref{SL-pb} with $\lambda=i\alpha$ and $\alpha\in\mathbb{R}$. Since the unique solution to \eqref{F} does not depend on $\theta$, see \eqref{uz-explicit}, we have $u_z(\rho)= \widehat{u}_{z, 0}(\rho)$ and necessarily $\widehat{u}_{z, k}\equiv 0$ for any $k\in \mathbb{Z}\setminus \{0\}$. Thus, we conclude that $W_k\equiv0$ for any $k\in \mathbb{Z}\setminus \{0\}$ and so $\lambda=i\alpha$ with $\alpha\in \mathbb{R}$ is not an eigenvalue.
	\end{proof}

\subsection{Stability with still outer cylinder.}\label{sec:unrestricteduniqueness}
	
We investigate here the existence of non-partially-invariant solutions to \eq{SNSEcyl}-\eqref{nsstokes0}$_1$ by perturbing $(\U_{\alpha, \beta}, \P_{\alpha, \beta})$, as defined in \eq{tcg}. We analyze the case when the perturbation has finite energy and we show that, if $\alpha$ and $\beta$ are sufficiently small, then the only allowed perturbation is the trivial one. To this end, we consider the space of divergence-free
vector fields
	$$
	\begin{array}{cc}
		V_0(\Omega):=\left\{v\in H^1_0(\Omega) \ , \ \nabla\cdot v= 0\text{ in }\Omega
		\right\}\, ,
	\end{array}
	$$ in which the Poincaré inequality
	\begin{equation}\label{poincare-const}
		\Lambda (\Omega)\| v\|^2_{L^2(\Omega)}\le\|\nabla v\|^2_{L^2(\Omega)}\qquad\mbox{with}\qquad
		\Lambda (\Omega)=\inf_{v\in V_0(\Omega)\setminus \{0\} }\frac{\|\nabla v\|^2_{L^2(\Omega)}}{\| v\|^2_{L^2(\Omega)}}
	\end{equation}
	holds for some $\Lambda(\Omega)>0$. Let us introduce a constant measuring the magnitude of the flow $(\U_{\alpha,\beta}, \P_{\alpha, \beta})$ when the outer cylinder is still, namely,
	\begin{equation}\label{M-alphabeta-new}
		M_{\rm out}(\alpha, \beta):= \left[\alpha^2 \frac{R_2^4}{(R_2^2 -R_1^2)^2R_1^2} + \frac{\beta^2}{16} \left(\frac{R_2^2-R_1^2}{R_1\log(R_2^2/R_1^2)}-R_1  \right)^2\right]^{\frac{1}{2}}\qquad\forall \ \alpha,\beta\in\R.
	\end{equation}
	Allowing a {\em finite-energy perturbation of the infinite-energy solution} \eqref{tcg}, we prove the following stability result:
	
	\begin{theorem}\label{uniqueness}
		Fix $\alpha,\beta \in \mathbb{R}$.  Let  $\Lambda(\Omega)$ and $M_{\rm out}(\alpha, \beta)$  be as in \eqref{poincare-const} and
		\eqref{M-alphabeta-new}, respectively. Assume that $(v, q)\in V_0(\Omega)\times L^2(\Omega)$ is such that $(\U_{\alpha,\beta} + v , \P_{\alpha,\beta} + q)$  solves \eqref{SNSEcyl}-\eqref{nsstokes0}$_1$. If
		\begin{equation}\label{small-cond-theo}
			M_{\rm out}\left(\alpha, \beta \right) < \Lambda (\Omega)\,,
		\end{equation} then $(v,q)\equiv(0,0)$ in $\Omega$. In particular, the result holds if
$$M_{\rm out}(\alpha, \beta) < \max\left\{\frac{\pi^2}{2R_2^2},\frac{8}{(R_2^2 - R_1^2)\log (R_2/R_1)}\right\}.
$$
	\end{theorem}
	\begin{proof} Since $(\U_{\alpha,\beta}, \P_{\alpha, \beta})$ solves \eqref{SNSEcyl}-\eqref{nsstokes0}$_1$,
		the perturbation $(v,q)\in V_0(\Omega)\times L^2(\Omega)$ satisfies
		\begin{equation}\label{perturbNS}\begin{cases}
				-\Delta v + v \cdot \nabla v +\U_{\alpha,\beta} \cdot \nabla v  + v \cdot \nabla \U_{\alpha,\beta} + \nabla q =0, \quad \nabla\cdot v =0 \quad \text{in} \quad \Omega, \\[5pt]
				v = 0\quad \text{on}\quad \Gamma_1\, ,\qquad v =0\quad \text{on}\quad \Gamma_2\, ,
			\end{cases}
		\end{equation}and, as mentioned above, it is a classical solution to \eqref{perturbNS}.
		Taking the $L^2(\Omega)$-scalar product of \eqref{perturbNS} with $v$, integrating by parts and using the divergence-free condition then yields
		\begin{equation}\label{dirnorm-v}
			\int_\Omega |\nabla v|^2  = - \int_{\Omega}\left(v\cdot \nabla \U_{\alpha,\beta}\right) \cdot v.
		\end{equation}
		After writing in cylindrical coordinates
		\begin{equation}
			\int_\Omega (	v\cdot \nabla \U_{\alpha,\beta} )\cdot v = \int_{(R_1,R_2)\times (0,2\pi)\times \mathbb{R}}\rho v^TA(\rho)v
		\end{equation}
		with the symmetric matrix
		\begin{equation}\label{Arho}
			A(\rho)=\left( \begin{matrix}
				0& \frac{\alpha}{2}\left(\frac{\U^C(\rho)}{\rho}-\frac{d \U^C(\rho)}{d\rho}\right)&-\frac{\beta}{2}\frac{d \U^P(\rho)}{d \rho}\\[5pt]
				\frac{\alpha}{2}\left(\frac{\U^C(\rho)}{\rho}-\frac{d \U^C(\rho)}{d\rho}\right)&0&0\\[5pt]
				-\frac{\beta}{2}\frac{d \U^P(\rho)}{d \rho}&0&0
			\end{matrix}\right).
		\end{equation}
		We know that there exists $\Upsilon(\rho)\in \mathbb{R}$ such that, for any $\rho\in [R_1,R_2]$, the quadratic form $v^T A(\rho) v $ is bounded from above by $\Upsilon(\rho) |v|^2 $. The optimal $\Upsilon(\rho)$ is the maximum eigenvalue of the matrix $A(\rho)$, namely,
		\begin{equation*}
				0<	\Upsilon(\rho)= \left[\alpha^2 \frac{R_1^2R_2^4}{(R_2^2 -R_1^2)^2\rho^4} + \frac{\beta^2}{16} \left(\rho -\frac{R_2^2-R_1^2}{\rho \log (R_2^2/R_1^2)}  \right)^2\right]^{\frac{1}{2}}.
		\end{equation*}
		Some lengthy computations show that
		$$
		h(\rho):=\rho-\frac{R_2^2-R_1^2}{\rho\log(R_2^2/R_1^2)}\mbox{ is convex},\quad h(R_1)<0<h(R_2)<|h(R_1)|\, .
		$$
		Hence, since $\rho\mapsto\rho^{-4}$ is decreasing, we infer that $\Upsilon(\rho)\le\Upsilon(R_1)=M_{\rm out}(\alpha,\beta)$ for all
		$\rho\in[R_1,R_2]$, with $M_{\rm out}(\alpha, \beta)$ given by \eqref{M-alphabeta-new}. Combining this uniform bound with \eqref{dirnorm-v} and \eqref{poincare-const} implies
		\begin{equation}\label{ineq-gradL2diff-new}
			\int_{\Omega}|\nabla v|^2 \leq \int_{(R_1,R_2)\times (0,2\pi)\times \mathbb{R}}\Upsilon(\rho) \rho |v|^2 \leq M_{\rm out}(\alpha, \beta)\int_{\Omega}|v|^2 \leq \frac{M_{\rm out}(\alpha, \beta)}{ \Lambda(\Omega)} \int_\Omega |\nabla v|^2.
		\end{equation}
		
		We see from \eqref{ineq-gradL2diff-new} that, if
		\eqref{small-cond-theo} holds, then $\nabla v \equiv 0$ in $\Omega$ and, due to the homogeneous boundary conditions, $v\equiv 0$ in $\Omega$. An analogous argument also gives that $ q\equiv0$ in $\Omega$.\par
		The second statement follows directly from the (just proved) first statement provided that
		\begin{equation}\label{boundPoincare}
			\Lambda (\Omega)\ge\frac{\pi^2}{2R_2^2}\quad\mbox{and}\quad\Lambda(\Omega)\ge\frac{8}{(R_2^2 - R_1^2)\log (R_2/R_1)}\, ,
		\end{equation}
		where $\Lambda(\Omega)$ is as in \eqref{poincare-const}. For the first inequality, we show that
		\begin{equation}\label{sobolevconstants11}
			\|v\|_{L^2(\Omega)}^2\leq \frac{2R_2^2}{\pi^2}\, \|\nabla v\|_{L^2(\Omega)}^2\hspace*{5mm}\forall v\in H_0^1(\Omega),
		\end{equation}which, in particular, holds for all $ v\in V_0(\Omega)\subset H^1_0(\Omega)$.
		Let us first prove \eqref{sobolevconstants11} for scalar functions $v\in H_0^1(\Omega)$. Since $\Omega\subset Q:=(-R_2,R_2)^2\times\R$, by
		trivial extension of any $v\in H_0^1(\Omega)$ to all $Q$ (the circumscribed unbounded parallelepiped), we have that
		$\Lambda(\Omega)\ge\Lambda(Q)$. By density, we may restrict our attention to functions $v\in H^1_0(\Omega)$ whose support is contained in
		$[-R_2,R_2]^2\times I$, where $I$ is an open bounded interval. Since $\cos(\tfrac{\pi x}{2R_2})\cos(\tfrac{\pi y}{2R_2})$ is an eigenfunction
		of $-\Delta$ corresponding to the least Dirichlet eigenvalue in the square $(-R_2,R_2)^2$, the resulting Poincaré inequality is
		$$
		\|w\|_{L^2((-R_2,R_2)^2)}^2\le\frac{2R_2^2}{\pi^2}\big(\|w_x\|_{L^2((-R_2,R_2)^2)}^2+\|w_y\|_{L^2((-R_2,R_2)^2)}^2\big)\qquad
		\forall w\in H^1_0((-R_2,R_2)^2)\, .
		$$
		Therefore, for any $v$ in the dense subspace characterized above and for a.e.\ $z\in I$ we have
		$$
		\int_{(-R_2,R_2)^2}v(x,y,z)^2dxdy\le\frac{2R_2^2}{\pi^2}\int_{(-R_2,R_2)^2}\big(v_x(x,y,z)^2+v_y(x,y,z)^2\big)dxdy
		$$
		which, after integration for $z\in I$ and by density, proves \eq{sobolevconstants11} for scalar functions.\par
		If $v=(v_1,v_2,v_3)\in H_0^1(\Omega)$ is a vector field, then $v_i\in H_0^1(\Omega)$ for $i=1,2,3$ and, by \eq{sobolevconstants11} for scalars,
		$$
		\begin{aligned}
			\|v\|^2_{L^2(\Omega)} & = \|v_{1}\|^2_{L^2(\Omega)} + \|v_{2}\|^2_{L^2(\Omega)}+ \|v_3\|^2_{L^2(\Omega)} \\
			& \leq \frac{2R_2^2}{\pi^2} \left( \| \nabla v_{1}\|^{2}_{L^2(\Omega)} + \| \nabla v_{2}\|^{2}_{L^2(\Omega)}+ \| \nabla v_3\|^{2}_{L^2(\Omega)}\right)
			=\frac{2R_2^2}{\pi^2}\| \nabla v \|^{2}_{L^2(\Omega)} \, ,
		\end{aligned}
		$$
		thereby proving \eqref{sobolevconstants11} also for vector fields. Hence, \eqref{boundPoincare}$_1$ follows.\par
		In order to prove \eqref{boundPoincare}$_2$, we argue as for \eqref{sobolevconstants11} and we restrict to scalar functions $v\in C^\infty_c(\Omega)$ that can be extended trivially outside
		$\Omega$ obtaining $v\in C^\infty_c(\R^3)$. Passing to cylindrical coordinates and due to the homogeneous boundary conditions on both $\Gamma_1$ and $\Gamma_2$, we have 
	for any fixed $(\theta,z)\in[0,2\pi)\times\R$ 
	\begin{eqnarray*}
		v(\rho,\theta, z)=
		\int_{R_1}^\rho\partial_\rho v(r, \theta, z) dr, \qquad
		v(\rho, \theta,z)=
		-\int_\rho^{R_2}\partial_\rho v(r, \theta, z) dr\quad\forall\rho\in(R_1,R_2).
	\end{eqnarray*}
	Combining both equalities and using H\"older's inequality, we find for all $\rho\in(R_1,R_2)$ 
	\begin{eqnarray*}
		|v(\rho,\theta,z)|^2&\le&
	\frac{1}{4}	\left(\int_{R_1}^{R_2}\left| \partial_\rho v(r, \theta, z)\right|dr\right)^2\le	\frac{1}{4}\left(\int_{R_1}^{R_2} \frac{dr}{r} \right)\left(\int_{R_1}^{R_2}\left|\partial_\rho v(r, \theta, z)\right|^2r dr\right)\, \\&=& \frac{\log\left(R_2/R_1\right)}{4} \left(\int_{R_1}^{R_2}\left|\partial_\rho v(r, \theta, z)\right|^2r dr\right).
	\end{eqnarray*}
	After multiplying both sides by $\rho$ and integrating with respect to $(\rho, \theta, z)\in (R_1,R_2)\times (0,2\pi)\times \mathbb{R}$, we transform back to cartesian coordinates and obtain the bound
	\begin{equation}
		\int_\Omega |v|^2\leq \frac{(R_2^2 - R_1^2) \log\left(R_2/R_1\right) }{8}\int_\Omega |\nabla v|^2,
	\end{equation}
		thereby proving \eqref{boundPoincare}$_2$.		
	\end{proof}
	
	We know from \cite[p.226]{temam2001navier} that if $( \U_{\alpha,0}+v,  \P_{\alpha,0}+q)$  solves \eqref{SNSEcyl}-\eqref{nsstokes0}$_1$
	and if $v$ is either $z$-independent or its vertical component $v_z$ vanishes
	as $|z|\to\infty$, then, $(v,q)\equiv(0,0)$ regardless of the value of $\alpha$. Theorem \ref{theo-tcg} generalizes this stability result. In fact, if one requires periodicity on $z$, or independence of $z$, or $u_z\rightarrow 0$ as $|z|\to\infty$, the constant $\beta\in \mathbb{R}$ in the
	spiral Poiseuille flows \eqref{tcg} must necessarily be $\beta=0$ and $( \U_{\alpha,0},  \P_{\alpha, 0})$ coincides with the circular Couette flows $(\alpha{\mathcal U^C}e_\theta,\alpha^2{\P^C})$, see \eqref{tc1}, and the smallness condition
	\eqref{small-cond-theo} ensuring stability reduces to
	\begin{equation*}
		|\alpha| < \frac{(R_2^2-R_1^2)R_1}{R_2^2}\Lambda(\Omega).
	\end{equation*}
	We also recall from \cite{temam2001navier,zeidlerIV} that $(\alpha{\mathcal U}^Ce_\theta,\alpha^2{\P^C})$ is the unique axially symmetric and $z$-periodic solution to \eqref{SNSEcyl}-\eqref{nsstokes0}$_1$ for sufficiently small $|\alpha|$.
	
	In the appendix, we complement Theorem \ref{uniqueness} with some qualitative comments on the behaviour of the stability condition \eqref{small-cond-theo} as $R_1$ and $R_2$ vary, see Proposition \ref{R1R2vary}.
	
	\subsection{The case of still inner cylinder.}\label{sec:velocityoutercylinder}
	
	We assume here that the inner cylinder $\Gamma_1$ is still while the outer cylinder $\Gamma_2$ is rotating with constant angular velocity
	$\alpha/R_2$. Then, we need to replace \eqref{nsstokes0}$_1$ by \eqref{nsstokes0}$_2$.
	With the very same proof, instead of Theorem \ref{theo-tcg} we obtain:
	
	\begin{theorem}\label{theo-tcg2}
		Let $({\widetilde\U^C},{\widetilde\P^C})$ be as in \eqref{tc2}, let $\U^P$ be as in \eqref{Ualphabeta}. Then the spiral Poiseuille flows
		\begin{equation} \label{tcg2}
			\begin{aligned}
				&\widetilde\U_{\alpha,\beta}(\rho) := \alpha\widetilde\U^C(\rho)e_\theta+ \beta\U^P(\rho)e_z ,\\[5pt]
				& \widetilde\P_{\alpha,\beta}(\rho,z) := \alpha^2\widetilde\P^C(\rho)+ \beta z,\end{aligned}  \qquad \text{for} \quad  (\rho, z)\in [R_1,R_2]\times \mathbb{R}\,
		\end{equation}
		are the only partially-invariant solutions  to \eqref{SNSEcyl}-\eqref{nsstokes0}$_2$ as $ \beta$ varies in $\mathbb{R}$.
	\end{theorem}
	
	As expected, \eqref{tcg} and \eqref{tcg2} coincide when $\alpha=0$ (both cylinders are still).\par
	For the stability result, the equation to be considered is again \eqref{perturbNS}, as in the case of still outer cylinder.
	Therefore, neither the Poincaré constant \eqref{poincare-const} nor the space $V_0(\Omega)$ change.
	However, although the symmetric matrix $\widetilde A(\rho)$ in \eqref{Arho} is formally the same, now $\widetilde\U_{\alpha,\beta}$ is as in \eqref{tc2}. Therefore,
	$$
	\widetilde\Upsilon(\rho)=\left[\alpha^2\frac{R_1^4R_2^2}{(R_2^2-R_1^2)^2\rho^4}
	+\frac{\beta^2}{16}\left(\rho-\frac{R_2^2-R_1^2}{\rho\log(R_2^2/R_1^2)}\right)^2\right]^{\frac{1}{2}},
	$$
	and the same computations then show that $\widetilde\Upsilon(\rho)\le\widetilde\Upsilon(R_1)=M_{\rm in}(\alpha,\beta)$ for all $\rho\in[R_1,R_2]$, where
	\begin{equation}\label{M-alphabeta-new2}
		M_{\rm in}(\alpha,\beta):= \left[\alpha^2 \frac{R_2^2}{(R_2^2 -R_1^2)^2}+ \frac{\beta^2}{16}
		\left(\frac{R_2^2-R_1^2}{R_1\log(R_2^2/R_1^2)}-R_1  \right)^2\right]^{\frac{1}{2}}\qquad\forall \ \alpha,\beta\in\R,
	\end{equation}
	which replaces the constant $M_{\rm out}(\alpha,\beta)$ in \eqref{M-alphabeta-new}. Then, the counter-part of Theorem \ref{uniqueness} reads
	
	\begin{theorem}\label{uniqueness2}
		Fix $\alpha,\beta\in\mathbb{R}$. Let $\Lambda(\Omega)$ and $M_{\rm in}(\alpha,\beta)$ be as in \eqref{poincare-const} and \eqref{M-alphabeta-new2}, respectively. Assume that $(v,q)\in V_0(\Omega)\times L^2(\Omega)$
		is such that $(\widetilde\U_{\alpha,\beta} + v , \widetilde\P_{\alpha,\beta} + q)$  solves \eqref{SNSEcyl}-\eqref{nsstokes0}$_2$. If
		\begin{equation}\label{small-cond-theo2}
			M_{\rm in}(\alpha,\beta) < \Lambda (\Omega)\,,
		\end{equation}
		then $(v,q)\equiv(0,0)$ in $\Omega$.
	\end{theorem}
	
	Notice that $M_{\rm out}(\alpha,\beta)<M_{\rm in}(\alpha,\beta)$ whenever $\alpha\neq0$ but, in fact, the two conditions
	\eqref{small-cond-theo} and \eqref{small-cond-theo2} coincide in terms of the angular velocities $w_i=\alpha/R_i$ ($i=1,2$).
	When $\beta=0$, both \eqref{small-cond-theo} and \eqref{small-cond-theo2} become
	$$
	|w|<\left(1-\frac{R_1^2}{R_2^2}\right)\Lambda(\Omega)\, .
	$$

	\section{Boundary conditions involving the vorticity}\label{boundaryvortex}

    \subsection{Derivation and connection with Navier slip conditions.}
	We start this section with an analysis of the vorticity associated with the spiral Poiseuille flows \eqref{tcg} and \eqref{tcg2}, which are the only partially-invariant solutions to \eqref{SNSEcyl}-\eqref{nsstokes0}, see Theorems \ref{theo-tcg} and \ref{theo-tcg2}. We give full details only for the first case (with still outer cylinder), the other case being similar. This vorticity reads
	\begin{equation}\label{vorticity-tcg}
		\omega_{\alpha,\beta}:= \nabla \times  \U_{\alpha,\beta}=\frac{\beta}{4}\left(\frac{R_2^2-R_1^2}{\log(R_2/R_1)\rho}
		-2\rho\right)e_\theta -\frac{\alpha 2R_1}{R_2^2-R_1^2} e_z,
	\end{equation}
	which is constant if and only if $\beta=0$. Let us focus on its azimuthal component
	$$
	\frac{\beta}{4}\left(\frac{R_2^2-R_1^2}{\log(R_2/R_1)\rho}-2\rho\right)e_\theta.
	$$	
	It vanishes at $R_0$, see \eqref{R0},  independently of $\beta\neq 0$:
    $\omega_{0,\beta}(R_0,\theta,z)=0$ for any $(\theta,z)\in[0,2\pi)\times\R$, over the
    intermediate cylinder (of radius $R_0$) where the 
the velocity $\mathcal{U}_{0,\beta}$ attains its maximum value (in modulus), see Figure \ref{figvortex}.\par	
	In this section we partially replace the boundary conditions \eqref{nsstokes0}$_1$ with conditions involving $\nabla \times u$. To determine them, we use the explicit form \eqref{vorticity-tcg}. Following \cite[Section 2.1]{concaenglish},
	we consider the boundary value problem
	\begin{equation}\label{ns-vorti}
		\mbox{\eqref{SNSEcyl}}\, ,\qquad u\cdot \nu =0, \quad (\nabla \times u)\times \nu = \omega_{\alpha,\beta}\times \nu \quad \text{on}\quad \Gamma_1\, ,\qquad u =0\quad \text{on}\quad \Gamma_2\, .
	\end{equation}
From a mathematical point of view, the boundary condition at $\Gamma_1$ for the vorticity naturally appears when writing down the bilinear form in the weak formulation of \eqref{SNSEcyl}, see Theorem \ref{theo:unrestuniquevortouter}, using the identity $-\Delta u =\nabla\times (\nabla \times u)$ for divergence-free vector fields. The physical interpretation of the boundary conditions at $\Gamma_1$ is related to that of the Navier slip conditions. Indeed, in the cylindrical configuration we are considering, and using $u\cdot \nu =u_\rho=0$ on $\Gamma_1$, the following relation holds
	\begin{equation}\label{rel-bc}
		(	\nabla \times u) \times \nu = 2 \left[D(u)\nu\right]_\tau + \frac{2}{R_1}u_\theta e_\theta  \quad \text{on} \quad \Gamma_1,
	\end{equation}
	where the symmetric part of the gradient $$D(u)= \frac{\nabla u + \nabla u^T}{2} $$ is the strain-rate tensor and  $[f]_\tau=  f- (f\cdot \nu) \nu$ denotes the orthogonal projection onto the tangent plane to $\Gamma_1$. The second term on the right-hand side of \eqref{rel-bc} represents a geometric friction along the azimuthal direction due to the (horizontal) curvature of the cylinder $\Gamma_1$. As the radius $R_1$ increases and the boundary tends to be locally flat, the friction decreases and the partial slip tends to perfect slip in the azimuthal direction, whereas the fluid continues to perfectly slip in the vertical direction. We stress that the different behavior depending on the different tangential directions is a consequence of the anisotropy of the cylindrical geometry. Introducing the $3\times3$ anisotropic friction matrix  $$\mathcal{F}=\left(\begin{matrix}
		0&0&0\\0&2/R_1&0\\
		0&0&0
	\end{matrix}\right)$$ permits to recast \eqref{rel-bc} as
	\begin{equation}\label{Navier-bc}
		(	\nabla \times u)\times \nu = 2 \left[D(u)\nu\right]_\tau + \mathcal{F}\left[u\right]_\tau  \quad \text{on} \quad \Gamma_1,
	\end{equation}
	and the boundary conditions at $\Gamma_1$ in \eqref{ns-vorti} as inhomogeneous Navier slip boundary conditions with anisotropic friction. More precisely, we obtain
	\begin{equation}
		u\cdot \nu =\mathcal{U}_{\alpha,\beta}\cdot \nu =0, \quad 2 \left[D(u)\nu\right]_\tau + \mathcal{F}\left[u\right]_\tau = 2 \left[ D(\mathcal{U}_{\alpha,\beta})\nu\right]_\tau + \mathcal{F}\left[\mathcal{U}_{\alpha,\beta}\right]_\tau \quad \text{on}\quad \Gamma_1,\,
	\end{equation}with boundary data written in terms of the spiral Poiseuille flows $\mathcal{U}_{\alpha, \beta}$.
	Passing to cylindrical coordinates, the boundary conditions in \eqref{ns-vorti} then read
	\begin{equation}\label{bc-cylindrical}
		\begin{aligned}
			&\qquad u_\rho (R_1, \theta, z)=0, \qquad \frac{1}{R_1}u_\theta(R_1, \theta, z)+ \frac{\partial u_\theta}{\partial \rho}  (R_1, \theta, z)=(\omega_{\alpha, \beta})_z(R_1, \theta, z), \\ & \frac{\partial u_z}{\partial \rho} (R_1, \theta, z)= -(\omega_{\alpha, \beta})_\theta (R_1, \theta, z), \quad
			u_\rho (R_2, \theta, z)= u_\theta (R_2, \theta, z)= u_z (R_2, \theta, z)=0,
		\end{aligned}
	\end{equation}for $(\theta, z)\in [0,2\pi)\times \mathbb{R}$. Note that, differently from \eq{nsstokes0}$_1$, here $u_\theta$ and $u_z$ satisfy, respectively, a Robin and a Neumann boundary condition on $\Gamma_1$.
    
\subsection{Explicit derivation of partially-invariant solutions.} In this subsection, we find all partially-invariant solutions to \eqref{ns-vorti} in the sense of Definition \ref{kindsolutions} only in the case with a still outer cylinder, the other case being similar.
    
	\begin{theorem}\label{theo-tcg-vorti}
		Let $({ \U^C},{\P^C})$ be as in \eqref{tc1}. Then the spiral Poiseuille-Couette flows
		\begin{equation} \label{tcg-vorti}\begin{aligned}
				& \U_{\alpha,\beta,\gamma}(\rho) := \alpha\U^C(\rho)e_\theta+  \left[ \frac{\gamma}{4}\left(\rho^2 - R_2^2 -2R_1^2 \log (\rho/R_2)\right)  -\beta R_1(\omega_{0,1})_\theta(R_1)\log (\rho/R_2) \right] e_z,\\[5pt]
				&  \P_{\alpha,\gamma}(\rho,z) := \alpha^2\P^C(\rho)+ \gamma z, \qquad for \quad (\rho,z) \in[R_1,R_2]\times \mathbb{R}\,\end{aligned}
		\end{equation}
		are the only partially-invariant solutions  to \eqref{ns-vorti} as $\gamma$ varies in $\mathbb{R}$.
	\end{theorem}
	\begin{proof}
		The proof has some common parts with the proof of Theorem \ref{theo-tcg}.
		Let $u$ be a partially-invariant solution to \eqref{ns-vorti}.  \\
		\underline{Step 1:} Combining the divergence-free property with the boundary conditions for $u_\rho$ in \eqref{bc-cylindrical} yields $u_\rho(\rho, \theta, z)\equiv 0 $, see \eqref{incom-cylin}.
		Then, after writing the equations in \eqref{SNSEcyl} in the class of partially-invariant solutions with $u_\rho\equiv 0$, see \eqref{uno}-\eqref{tres},  we obtain
		\begin{equation}\label{sinB-2}
			\rho u_z(\rho,\theta)\, \frac{\partial u_{\theta}}{\partial z}(\rho,z)=\rho\frac{\partial^2 u_{\theta}}{\partial\rho^2}(\rho,z) +
			\frac{\partial u_{\theta}}{\partial\rho}(\rho,z)+\rho\frac{\partial^2u_{\theta}}{\partial z^2}(\rho,z)-\frac{u_\theta(\rho,z)}{\rho}
		\end{equation}for all $(\rho, \theta, z)\in(R_1,R_2)\times (0, 2\pi)\times \mathbb{R}$,
		which in turn leads to the alternative \eqref{either}.
		
		\underline{Step 2:} we prove that, necessarily, $(u,p)=(\U_{\alpha,\beta,\gamma}, \P_{\alpha,\gamma})$ given by \eqref{tcg-vorti} as $\gamma$ varies in  $\mathbb{R}$. To this end, we distinguish two cases.\par
		(a) If the first alternative in \eqref{either} holds, then \eqref{sinB-2} becomes
		$$
		\rho^2\frac{d^2 u_{\theta}}{d\rho^2}(\rho)+\rho\frac{du_{\theta}}{d\rho}(\rho)-u_\theta(\rho)=0\quad\text{in}\quad (R_1,R_2)
		$$
		which, complemented with the boundary conditions \eqref{bc-cylindrical}, gives again $u_\theta(\rho)= \alpha\U^C(\rho)e_\theta$ with $\U^C$ as in \eqref{tc1}.
		Injecting this into \eqref{uno} and arguing similarly to the proof of Theorem \ref{theo-tcg}, we infer that there exists $\gamma\in\R$ such that
		\begin{equation}\label{sepa-pres2}
			p(\rho, \theta, z)=\alpha^2\P^C(\rho)+\gamma z
		\end{equation}with $\mathcal{P}^C$ as in \eqref{tc1}.
		Recalling \eqref{utheta}, then \eqref{tres} reduces to
		\begin{equation}\label{eq-uz-cyl}
			\frac{\partial^2 u_z}{\partial\rho^2}(\rho,\theta)+\frac{1}{\rho}\frac{\partial u_z}{\partial\rho}(\rho,\theta)+
			\frac{1}{\rho^2}\frac{\partial^2 u_z}{\partial\theta^2}(\rho,\theta)+
			\frac{\alpha R_1 }{R_2^2 -R_1^2} \left(1-\frac{R_2^2}{\rho^2}\right)\, \frac{\partial u_z}{\partial\theta}(\rho,\theta)=\gamma .
		\end{equation}
		For $2\pi$-periodic functions $\theta\mapsto v(\rho, \theta)$, we restrict ourselves to the class
		\begin{equation}\label{V*}
			V_*=\left\{v\in H^1((R_1,R_2)\times [0,2\pi)) \ | \ v(R_2,\theta)=0  \text{ for $\theta\in[0,2\pi)$}\right\},
		\end{equation}which is a closed subspace of $H^1((R_1,R_2)\times [0,2\pi))$. We claim that \eqref{eq-uz-cyl} admits a unique solution in the space $V_*$. Indeed, after multiplying \eqref{eq-uz-cyl} by $-\rho$, its weak formulation reads
		\begin{equation}\label{var-form2}
			B(u_z, v)= F_*(v) \ \text{for any} \ v\in V_*,
		\end{equation}
		where $B(u_z,v)$ is the bilinear form \eqref{B} on $V_*$ and $$F_*(v)=\beta R_1(\omega_{0,1})_\theta(R_1)\int_{0}^{2\pi}v(R_1,\theta)d\theta -\gamma\int_{(R_1,R_2)\times (0,2\pi)} \rho v(\rho, \theta) d\rho d\theta$$ is a linear functional on $V$. Thanks to the Hölder inequality and the embedding $H^1 (R_1,R_2)\subset C^0([R_1,R_2])$, we have that $|F(v)|\leq C_1\|v\|_{V_*}$ and,
		since $\rho\in [R_1,R_2]$, that $$|B(u_z, v)|\leq C_2 \|u_z\|_{V_*}\|v\|_{V_*}$$ for some constants $C_1=C_1(R_1,R_2,\beta, \gamma)$ and $C_2=C_2(R_1,R_2, \alpha)$. By arguing
        as in the proof of Theorem \ref{theo-tcg}, we find that
	$B$ is coercive (or $V_*$-elliptic). Then, for any $(\alpha, \beta,\gamma)\in \mathbb{R}^3$, the Lax-Milgram Theorem yields the existence and uniqueness of the solution to \eqref{eq-uz-cyl} in $V_*$. Seeking the solution $u_z$ as a function depending only on $\rho$, we find the system
	\begin{equation*}
		\begin{cases}
			\dfrac{d^2 {u_z}}{d\rho^2}(\rho) +\dfrac{1}{\rho}\dfrac{d{u_z}}{d\rho}(\rho) =\gamma \quad \text{in} \quad (R_1,R_2),\\[7pt]
			\dfrac{d u_z}{d\rho}(R_1)=-(\omega_{\alpha,\beta})_\theta(R_1), \quad   {u}_{z}(R_2)=0,
		\end{cases}
	\end{equation*}whose solution has the explicit expression
	\begin{equation}\label{uz-explicit-new}
		{u}_{z}(\rho)=  -\beta R_1(\omega_{0,1})_\theta(R_1)\log(\rho/R_2) + \frac{\gamma}{4}\left( \rho^2-R_2^2 -2R_1^2\log(\rho/R_2)\right).
	\end{equation}
	Thus, by uniqueness, we infer that \eqref{uz-explicit-new} is the unique solution to \eqref{eq-uz-cyl}.\par

	(b) The second alternative in \eqref{either} can be handled as in the proof of Theorem \ref{theo-tcg}, since it does not depend on the new boundary conditions \eqref{bc-cylindrical}.
\end{proof}

\begin{remark}
The components $\U_{0,0,\gamma}$ and 
$\U_{0,\beta,0}$ of $\U_{\alpha,\beta,\gamma}$ are usually derived after imposing, respectively, a constant pressure gradient along the $z$-axis and a translational motion of the cylinder, see \cite{Joseph}. As for $\U_{0,\beta}$ in the Dirichlet case, here $\U_{0,0,\gamma}$ is obtained through the constraint of partial invariance.  The new component $\U_{0,\beta,0}$ with respect to the Dirichlet case appears due to the vorticity conditions \eqref{ns-vorti}, which generates a sliding effect on the velocity at $\Gamma_1$, equivalent to its vertical translation.
Finally, note that $\alpha$ and $\beta$ represent the size of the boundary data and $\gamma$ can be seen as the mass density.
\end{remark}

If we assume that the inner cylinder $\Gamma_1$ is still, the counterpart of \eqref{ns-vorti} corresponding to this configuration reads
\begin{equation}\label{ns-vorti-outer}
	\mbox{\eqref{SNSEcyl}}\, ,\qquad u =0 \quad \text{on}\quad \Gamma_1\,, \qquad u\cdot \nu =0, \quad (\nabla \times u)\times \nu = \omega_{\alpha,\beta}\times \nu \quad \text{on}\quad \Gamma_2\, .
\end{equation}

With some minor changes in the proof, instead of Theorem \ref{theo-tcg-vorti} we obtain:

\begin{theorem}\label{theo-tcg-vorti-outer}
	Let $({ \widetilde\U^C},{\widetilde\P^C})$ be as in \eqref{tc2}. Then the spiral Poiseuille-Couette flows
	\begin{eqnarray}
			& \widetilde\U_{\alpha,\beta,\gamma}(\rho) := \alpha\widetilde\U^C(\rho)e_\theta+  \left[ \frac{\gamma}{4}\left(\rho^2 - R_1^2 -2R_2^2 \log (\rho/R_1)\right)  -\beta R_2(\omega_{0,1})_\theta(R_2)\log (\rho/R_1)\right] e_z,  &  \notag \\[5pt]
			 &\widetilde\P_{\alpha,\gamma}(\rho, z) := \alpha^2\widetilde\P^C(\rho)+ \gamma z \qquad \text{for} \quad  (\rho, z)\in [R_1,R_2]\times \mathbb{R}\, & 
            \label{tcg-vorti-outer} 
	\end{eqnarray}
	are the only partially-invariant solutions  to \eqref{ns-vorti-outer} as $\gamma$ varies in $\mathbb{R}$.
\end{theorem}

\subsection{Stability.}\label{unrestrictvorticity}
 
We discuss here the smallness conditions under which suitable perturbations of the spiral Poiseuille-Couette flows \eqref{tcg-vorti} and \eqref{tcg-vorti-outer} are admissible.
To obtain the counterparts of Theorems \ref{uniqueness} and \ref{uniqueness2} for the new boundary conditions, we need to show the existence of strictly positive curl-Poincaré constants, that is, we seek inequalities such as 
\begin{equation}\label{eq:genericPoincarecurl}
C \norm{v}^2_{L^2(\Omega)}\leq \norm{\nabla\times v}_{L^2(\Omega)}^2,    
\end{equation}
for some constant $C>0$ and any $v$ in some subspace of $H^1(\Omega)$. However, the next remark shows that it is not clear whether this is always possible.

\begin{remark}\label{contro}
	For $0<R_1<R_2$ and $K>0$, consider the 3D bounded cylindrical annulus
	$$
	\Omega=\left\{(x,y,z)\in\R^3;\, R_1^2<x^2+y^2<R_2^2,\, |z|<K\right\}\, .
	$$
	Consider the vector field
	$$
	u(x,y,z)=\frac{y}{x^2+y^2}e_1-\frac{x}{x^2+y^2}e_2+0e_3\, .
	$$
	Then $u\in C^\infty(\overline{\Omega})$ and
	$$
	\nabla\cdot u=0\mbox{ in }\Omega\, ,\qquad\nabla\times u=0\mbox{ in }\Omega\, ,\qquad u\cdot\nu=0\mbox{ on }\partial\Omega\, .
	$$
Therefore, under the mere impermeability boundary conditions, there are non-trivial divergence-free and curl-free vector fields, even in bounded domains. This implies that the constant $C$ in \eqref{eq:genericPoincarecurl} vanishes.
\end{remark}

Hence, on the one hand, in our setting one can expect that the curl-Poincaré constants $C$ in \eqref{eq:genericPoincarecurl} might vanish because the domain is unbounded and not simply-connected. On the other hand, we stress that the cases with still inner or outer cylinder will be significantly different. Since the former is simpler, we study it first.

\smallskip

\underline{\emph{Still inner cylinder}}. In this case, the space of perturbations of $\widetilde{\mathcal{U}}_{\alpha,\beta,\gamma}$ maintaining the boundary conditions in \eqref{ns-vorti-outer} reads
\begin{equation}\label{eq:underlineVtildedef} 
	\underline{\widetilde V}:= \left\{v\in H^1(\Omega), \nabla\cdot v=0 \text{ in } \Omega,  v=0 \text{ on } \Gamma_1, v\cdot \nu=0  \text{ on } \Gamma_2\right\},
\end{equation}and we prove:
\begin{lemma}\label{lem:poincarecurlouter}
	Let $\underline{\widetilde V}(\Omega)$ be the space defined in \eqref{eq:underlineVtildedef}. Then,
	\begin{equation}\label{eq:poincarecurl}
		\widetilde{\underline{\Lambda}}(\Omega)=\inf_{v\in \widetilde{\underline{V}}\setminus\{ 0\} }\frac{\norm{\nabla\times v}_{L^2(\Omega)}^2}{\norm{v}_{L^2(\Omega)}^2}\geq \inf_{v\in \widetilde{\underline{V}}\setminus\{ 0\} }\frac{\norm{\nabla v}_{L^2(\Omega)}^2}{\norm{v}_{L^2(\Omega)}^2}>0.
	\end{equation}
\end{lemma}
\begin{proof}
	Let $v\in \underline{\widetilde{V}}\setminus \{0\}$ with compact support in $z$. We choose $K>0$ so that $|z|>K$ implies $v=0.$
	Let $\Gamma^K$ be the part of the curve in the $(\rho,z)$-plane defined by the equation
	$$(z-K)^4+\frac{16}{(R_2-R_1)^4}\left(\rho-\frac{R_1+R_2}{2}\right)^4=1,$$ such that $\text{supp } \Gamma^K\subset \{\rho\in [R_1,R_2],z\in [K,K+1]\}$. Let $-\Gamma^K$ be the symmetric curve to $\Gamma^K$ with respect to the $\rho$ axis. Then, the curve $$\left\{\rho=R_1,z\in[-K,K]\right\}\cup\left\{\rho=R_2,z\in[-K,K]\right\}\cup \Gamma^K\cup -\Gamma^K $$ is the $C^3$-boundary of a 2D-domain $\Omega_{2D}^{K}$. Finally, we define $\Omega^K$ as the 3D-rotation of  $\Omega^K_{2D}$ with respect to the $z$-axis, namely
	\begin{equation}\label{eq:Omega^K}
		\Omega^K=\left\{(x,y,z):(\sqrt{x^2+y^2},z)\in \Omega^K_{2D}\right\}.
	\end{equation}
	The boundary of $\Omega^K$ is composed by three parts: $\Gamma_1^K$ contained in the inner cylinder $\Gamma_1$, $\Gamma_2^K$ contained in the outer cylinder $\Gamma_2$, and $\Gamma_3^K$ corresponding to the rotated auxiliary curves that join the inner and outer cylinders.
	Then, $v$ belongs to
	$$\underline{\widetilde V}^K(\Omega^K):= \left\{v\in H^1(\Omega^K), \nabla\cdot v=0 \text{ in } \Omega^K,  v=0 \text{ on } \Gamma_1^K\cup \Gamma_3^K, v\cdot \nu=0  \text{ on } \Gamma_2^K\right\}.$$
	
	Since $\Omega^K$ is $C^3$ and bounded, we can apply \cite[Lemma I.3.8]{girault1986finite} to get
	\begin{equation}\label{eq:girault}
		\norm{\nabla v}^2_{L^2(\Omega^K)}+\int_{\Gamma^K_2} (\mathcal{R}v,v)ds=\norm{\nabla\times v}^2_{L^2(\Omega^K)},
	\end{equation}
	where we have also used that $v$ is divergence-free and $v=0$ on $\Gamma_1^K \cup \Gamma_3^K$. Here $\mathcal{R}$ denotes the curvature tensor on the tangent plane to $\Gamma_2^K$, which is non-negative definite when $\Omega^K$ is convex close to $\Gamma_2^K$; see \cite[Theorem 3.1.1.1]{grisvard1985elliptic} for a more general result implying this equality. Hence the second term in \eqref{eq:girault} is non-negative, and therefore
	$$\norm{\nabla v}^2_{L^2(\Omega^K)}\leq \norm{\nabla\times v}^2_{L^2(\Omega^K)}.    $$
	
	Note that $v\equiv 0$ in $\Omega\setminus \Omega_K$, so we also have
	$$\norm{\nabla v}^2_{L^2(\Omega)}\leq\norm{\nabla\times v}^2_{L^2(\Omega)}.$$
	
	This inequality is independent of $K$. By density, we deduce that the last inequality holds for any $v\in \widetilde{\underline{V}}$. It follows that \eqref{eq:poincarecurl} holds.
\end{proof}
Next, we introduce a constant measuring the magnitude of the flow $(\widetilde\U_{\alpha,\beta,\gamma},\widetilde\P_{\alpha,\gamma})$, namely
\begin{equation}\label{eq:M-vorticity-outer}
	{M}_{\rm in}(\alpha,\beta,\gamma):= \max_{\rho\in[R_1,R_2]}\left[\alpha^2 \frac{R_1^4R_2^2}{(R_2^2 -R_1^2)^2\rho^4} +\left(\frac{1}{\rho}\big(\frac{\beta-\gamma}{4}R_2^2-\frac{\beta}{8}\frac{R_2^2-R_1^2}{\log(R_2/R_1)}\big)+\rho\frac{\gamma}{4} \right)^2\right]^{\frac{1}{2}}\!\!.
\end{equation}
We observe that ${M}_{\rm in}(0,0,0)=0$, and ${M}_{\rm in}(\alpha,0,0)=|\alpha|\frac{R_2}{R_2^2-R_1^2}$ for any $\alpha\in\mathbb{R}$. We prove the stability result:

\begin{theorem}\label{theo:unrestuniquevortouter}
	Fix $\alpha,\beta,\gamma\in\mathbb{R}$. Let $\underline{\widetilde{\Lambda}}(\Omega)$ and ${M}_{\rm in}(\alpha,\beta,\gamma)$ be as in \eqref{eq:poincarecurl} and \eqref{eq:M-vorticity-outer}, respectively. Assume that $(v,q)\in \underline{\widetilde V}(\Omega)\times L^2(\Omega)$ is such that $(\widetilde\U_{\alpha,\beta,\gamma}+v,\widetilde\P_{\alpha,\gamma}+q)$ solves \eqref{ns-vorti-outer}. If
	\begin{equation}\label{uniqueness-cond-vort-outer}
		{M}_{\rm in}(\alpha,\beta,\gamma)<\underline{\widetilde{\Lambda}}(\Omega),
	\end{equation}
	then $(v,q)\equiv (0,0)$ in $\Omega$.
\end{theorem}
\begin{proof}
	We proceed as in the proof of Theorem \ref{uniqueness}.
	Since $(\widetilde\U_{\alpha,\beta,\gamma},\widetilde\P_{\alpha,\gamma})$ solves \eqref{ns-vorti-outer}, the perturbation $(v,q)$ satisfies
	
	\begin{equation}\label{perturbNSoutervort}\begin{cases}
			-\Delta v + v \cdot \nabla v +\widetilde {U}_{\alpha,\beta,\gamma} \cdot \nabla v  + v \cdot \nabla \widetilde {U}_{\alpha,\beta,\gamma} + \nabla q =0, \quad \nabla\cdot v =0 \quad \text{in} \quad \Omega, \\[5pt]
			v = 0\quad \text{on}\quad \Gamma_1\, ,\qquad v\cdot \nu =0,\quad
			(\nabla\times v)\times \nu=0 \quad
			\text{on}\quad \Gamma_2\, ,
		\end{cases}
	\end{equation}and, as mentioned above, it is a classical solution to \eqref{perturbNSoutervort}.
	
	We recall that for any divergence-free $v$, it holds $-\Delta v=\nabla\times(\nabla \times v)$.
	Taking the $L^2(\Omega)$-scalar product of \eqref{perturbNSoutervort} with $v$, integrating by parts and using the divergence-free condition yields
	\begin{equation}\label{dirnorm-vort}
		\int_\Omega |\nabla\times v|^2  = \int_{\Gamma_2} ((\nabla\times v)\times \nu)\cdot v- \int_{\Omega}\left(v\cdot \nabla \widetilde{\U}_{\alpha,\beta,\gamma}\right) \cdot v=- \int_{\Omega}\left(v\cdot \nabla \widetilde{\U}_{\alpha,\beta,\gamma}\right) \cdot v.
	\end{equation}
	Using cylindrical coordinates,
	\begin{equation}\label{eq:nonlineartermenergy}
		\int_\Omega (	v\cdot \nabla \widetilde{\U}_{\alpha,\beta,\gamma} )\cdot v = \int_{(R_1,R_2)\times (0,2\pi)\times \mathbb{R}}\rho v^T\widetilde{A}(\rho)v
	\end{equation}
	with the symmetric matrix
	\begin{equation}\label{Arho2}
		\widetilde{A}(\rho)=\left( \begin{matrix}
			0& \tfrac{\alpha}{2}(\frac{ \widetilde{\U}^C}{\rho}-\frac{d  \widetilde{\U}^C}{d\rho})&	-\tfrac{1}{2}\big(\beta\frac{d\widetilde \U_{0,1,0}}{d\rho}+\gamma\frac{d\widetilde \U_{0,0,1}}{d\rho}\big)\\[5pt]
			\tfrac{\alpha}{2}(\frac{ \widetilde{\U}^C}{\rho}-\frac{d  \widetilde{\U}^C}{d\rho})&0&0\\[5pt]
			-\tfrac{1}{2}\big(\beta\frac{d\widetilde \U_{0,1,0}}{d\rho}+\gamma\frac{d\widetilde \U_{0,0,1}}{d\rho}\big)&0&0
		\end{matrix}\right).
	\end{equation}
	We know that there exists $\widetilde\Upsilon(\rho)\in \mathbb{R}$ such that, for any $\rho\in [R_1,R_2]$, the quadratic form $v^T \widetilde A(\rho) v $ is bounded from above by $\widetilde\Upsilon(\rho) |v|^2 $. The optimal $\widetilde\Upsilon(\rho)$ is the maximum eigenvalue of the matrix $\widetilde A(\rho)$, namely,
	\begin{equation*}\begin{aligned}
			0<	\widetilde\Upsilon(\rho)&= 
            \left[\alpha^2 \frac{R_1^4R_2^2}{(R_2^2 -R_1^2)^2\rho^4} +\left(\frac{1}{\rho}\bigg(\frac{\beta-\gamma}{4}R_2^2-\frac{\beta}{8}\frac{R_2^2-R_1^2}{\log(R_2/R_1)}\bigg)+\rho\frac{\gamma}{4} \right)^2\right]^{\frac{1}{2}}.\\[5pt]
		\end{aligned}
	\end{equation*}
	By \eqref{eq:M-vorticity-outer}, $\widetilde\Upsilon(\rho)\leq {M}_{\rm in}(\alpha,\beta,\gamma)$. Combining this bound with \eqref{dirnorm-vort} and \eqref{eq:nonlineartermenergy}, we deduce
	\begin{equation*}
		\int_\Omega |\nabla\times v|^2 \leq \int_{(R_1,R_2)\times (0,2\pi)\times \mathbb{R}} \widetilde{\Upsilon}(\rho)\rho |v|^2\leq {M}_{\rm in}(\alpha,\beta,\gamma)\int_\Omega |v|^2\leq \frac{{M}_{\rm in}(\alpha,\beta,\gamma)}{\underline{\widetilde{\Lambda}}(\Omega)}\int_\Omega |\nabla\times v|^2.
	\end{equation*}
	Hence, if \eqref{uniqueness-cond-vort-outer} holds, $\nabla\times v\equiv0$ in $\Omega$ and, since $\widetilde{\underline{\Lambda}}>0$ by Lemma \ref{lem:poincarecurlouter}, we infer $v\equiv0$ in $\Omega$. Then, one infers that also $q\equiv0$ in $\Omega$.\end{proof}

\underline{\emph{Still outer cylinder}}. 
In this case, the space of perturbations of ${\mathcal{U}}_{\alpha,\beta,\gamma}$ maintaining the boundary conditions in \eqref{ns-vorti} reads 
\begin{equation}\label{eq:underlineVdef} 
\underline{V}:= \left\{v\in H^1(\Omega), \nabla\cdot v=0 \text{ in } \Omega, v\cdot \nu=0 \text{ on } \Gamma_1, v=0 \text{ on } \Gamma_2\right\}.
\end{equation}
We emphasize that we cannot derive a strict inequality as in \eqref{eq:poincarecurl} since \eqref{eq:girault} becomes
\begin{equation*}
	\norm{\nabla v}^2_{L^2(\Omega^K)}+\int_{\Gamma^K_1} (\mathcal{R}v,v)ds=\norm{\nabla\times v}^2_{L^2(\Omega^K)},
\end{equation*}
and the boundary integral is negative due to the concavity of $\Omega^K$ close to $\Gamma_1^K$. This is where we see the analytical difference between imposing vorticity boundary conditions on the inner or outer cylinder. To overcome this challenge, we obtain the positivity of curl-Poincaré constants by imposing additional constraints.\par 
We prove it by adding the geometric constraint of thin annulii, namely, considering small ratios of the radii:
	\begin{lemma}\label{lem:poincarecurlinner}
		Let $\underline{ V}(\Omega)$ be the space defined in \eqref{eq:underlineVdef}. If $ {R_2}/{R_1}\in (1,e),$ then
		\begin{equation}\label{eq:poincarecurlinner}
			{\underline{\Lambda}}(\Omega)=\inf_{v\in {\underline{V}}\setminus\{ 0\} }\frac{\norm{\nabla\times v}_{L^2(\Omega)}^2}{\norm{v}_{L^2(\Omega)}^2}\geq \left(1-\log\left(R_2/R_1\right)\right)  \inf_{v\in {\underline{V}}\setminus\{ 0\} }\frac{\norm{\nabla v}_{L^2(\Omega)}^2}{\norm{v}_{L^2(\Omega)}^2}>0.
		\end{equation}
	\end{lemma}
    \begin{proof}
    Following the proof of Lemma \ref{lem:poincarecurlouter}, we arrive to the equation
\begin{equation}\label{eq:girault2}
	\norm{\nabla v}^2_{L^2(\Omega^K)}+\int_{\Gamma^K_1} (\mathcal{R}v,v)ds=\norm{\nabla\times v}^2_{L^2(\Omega^K)}.
\end{equation}
The second term is now negative, but we can take advantage of the cylindrical geometry we are working with. Indeed, the curvature tensor can be made explicit and, from its definition in \cite[Section 3.1.1.1]{grisvard1985elliptic}, we find that $$\int_{\Gamma_1^K}(\mathcal{R}v,v)ds= -\frac{1}{R_1}\int_{\Gamma_1^K}|v_\theta|^2 ds= -\int_{(0,2\pi)\times (-K,K)}|v_\theta(R_1, \theta, z)|^2 d\theta dz. $$ Due to the homogeneous boundary conditions on $\Gamma_2^K$, we argue as in the proof of Theorem \ref{uniqueness} and obtain that
	\begin{equation*}
-\int_{\Gamma_1^K}(\mathcal{R}v,v)ds=\int_{(0,2\pi)\times (-K,K)}|v_\theta(R_1, \theta, z)|^2 d\theta dz\leq \log\left(R_2/R_1\right)\|\nabla v\|^2_{L^2(\Omega^K)}
	\end{equation*}which, combined with  \eqref{eq:girault2}, implies that $$\left(1-\log\left(R_2/R_1\right)\right)\|\nabla v\|^2_{L^2(\Omega_K)}\leq \|\nabla \times v\|^2_{L^2(\Omega_K)},$$
    completing the proof of \eqref{eq:poincarecurlinner}.\end{proof}

    Without the thin annulus assumption, it is not clear whether   \eqref{eq:poincarecurlinner} continues to hold.
	Thanks to Lemma \ref{lem:poincarecurlinner}, the proof of Theorem \ref{theo:unrestuniquevortouter} can be mimicked to obtain an analogous stability result.
	Before stating it, we introduce a constant measuring the magnitude of the flow $(\U_{\alpha,\beta,\gamma},\P_{\alpha,\gamma})$:
	\begin{equation}\label{eq:M-vorticity}
		{M}_{\rm out}(\alpha,\beta,\gamma)\!:= \! \!\max_{\rho\in [R_1,R_2]}\left[\alpha^2 \frac{R_1^2R_2^4}{(R_2^2 -R_1^2)^2\rho^4} \!+\!\left(\frac{1}{\rho}
		\bigg(\frac{\beta-\gamma}{4}R_1^2
		-\frac{\beta}{8}\frac{R_2^2-R_1^2}{\log(R_2/R_1)}\bigg)
		+\rho\frac{\gamma}{4} \right)^2\right]^{\frac{1}{2}}\!\!.
	\end{equation}
	We observe that ${M}_{\rm out}(0,0,0)=0$, and ${M}_{\rm out}(\alpha,0,0)=|\alpha|\frac{R^2_2}{(R_2^2-R_1^2)R_1}$ for any $\alpha\in\mathbb{R}$. We prove:
	 \begin{theorem}\label{theo-uni-inner}
		Fix $\alpha,\beta,\gamma\in\mathbb{R}$ and take ${R_2}/{R_1}\in (1,e)$. Let $\underline{{\Lambda}}(\Omega)$ and ${M}_{\rm out}(\alpha,\beta,\gamma)$ be as in \eqref{eq:poincarecurlinner} and \eqref{eq:M-vorticity}, respectively. Assume that $(v,q)\in \underline{ V}(\Omega)\times L^2(\Omega)$ is such that $(\U_{\alpha,\beta,\gamma}+v,\P_{\alpha,\gamma}+q)$ solves \eqref{ns-vorti}. If
		\begin{equation}\label{uniqueness-cond-vort-outer2}
			{M}_{\rm out}(\alpha,\beta,\gamma)<\underline{{\Lambda}}(\Omega),
		\end{equation}
		then $(v,q)\equiv (0,0)$ in $\Omega$.
	\end{theorem}

The positivity of the curl-Poincaré constant can also be obtained assuming $z$-periodicity of the velocity vector fields, with no need of the geometric restriction $R_2/R_1\in (1, e)$. Obviously, since $\U_{\alpha,\beta,\gamma}$ in \eqref{tcg-vorti} are independent of $z$, they are also $z$-periodic. We fix a period $L>0$ and set $\Omega_L=\{(x,y,z)\in \mathbb{R}^3:R_1^2<x^2+y^2<R_2^2, z\in [-L/2,L/2)\}$. We define the space
\begin{equation}\label{eq:V_L}\begin{aligned}
		V_L:= \big\{&v\in H^1_{\rm loc}(\Omega), \nabla\cdot v=0 \text{ in } \Omega, v\cdot \nu=0 \text{ on } \Gamma_1, v=0 \text{ on } \Gamma_2, \\[2pt]& \ v(x,y,z)=v(x,y,z+L) \quad  \forall (x,y,z)\in \Omega\big\}
	\end{aligned}
\end{equation}
and we prove:

\begin{lemma}\label{lem:poincarecurlperiodic}
Let $V_L$ be the space defined in \eqref{eq:V_L}. Then,
 \begin{equation}\label{eq:Lambda_L}
		\Lambda_L(\Omega):=\inf_{v\in V_L\setminus\{ 0\} }\frac{\norm{\nabla\times v}_{L^2(\Omega_L)}^2}{\norm{v}_{L^2(\Omega_L)}^2}>0.
	\end{equation}
\end{lemma}
\begin{proof}
We construct a suitable compact subregion of $\Omega$, as in the proof of Lemma \ref{lem:poincarecurlouter}. Let $K>0$ be a real number, and $\Gamma^K$ be the part of the curve within the $(\rho,z)$-plane defined by
	$$(z-K)^4+\frac{16}{(R_2-R_1)^4}\left(\rho-\frac{R_1+R_2}{2}\right)^4=1$$ such that $\text{supp } \Gamma^K\subset \{\rho\in [R_1,R_2],z\in [K,K+1]\}$. Let $-\Gamma^K$ be the symmetric curve to $\Gamma^K$ with respect to the $\rho$ axis. Then, the curve $$\left\{\rho=R_1,z\in[-K,K]\right\}\cup\left\{\rho=R_2,z\in[-K,K]\right\}\cup \Gamma^K\cup -\Gamma^K $$ is the $C^3$-boundary of a 2D-domain $\Omega_{2D}^{K}$. Finally, we define $\Omega^K$ as the 3D-rotation of  $\Omega^K_{2D}$ with respect to the $z$-axis, namely
	\begin{equation}\label{eq:Omega^K2}
		\Omega^K=\left\{(x,y,z):(\sqrt{x^2+y^2},z)\in \Omega^K_{2D}\right\}.
	\end{equation}
	Take any $w\in V_L \setminus \{0\}$. We choose a smooth function $\chi_{K}(z)$ such that $0 \leq \chi_{K} \leq 1$,
	$\chi_{K} \equiv 1$ if $|z| \leq K$, and $\chi_{K} \equiv 0$ if $|z| \geq K+1$.
	Then, it is clear that $\chi_{K} w$ satisfies
	$\chi_{K} w \cdot \nu = 0$ on the boundary of $\Omega^{K+1}$, and
	$\chi_{K} w= 0$ on part of the boundary of $\Omega^{K+1}$.
	Therefore, we can apply \cite[Theorem A.1]{concaenglish} to deduce that there exists
	a constant $C_{\Omega^{K+1}}>0$ such that
	$$\norm{ \chi_{K} w}_{H^1(\Omega^{K+1})}\leq C_{\Omega^{K+1}} (\norm{\chi_{K} w}_{L^2(\Omega^{K+1})}+\norm{\nabla\times (\chi_{K} w)}_{L^2(\Omega^{K+1})}).$$
	Choosing $K=L/2$, we observe that $\Omega_L\subset \Omega^{L/2}\subset \Omega^{L/2+1}\subset \Omega _{L\lceil 1+4/L\rceil }$, where $\lceil s \rceil$ is the least integer greater than or equal to $s$. It follows that
	\begin{equation}\label{eq:concaappendixnorms}
		\begin{aligned}
			\norm{ w}_{H^1(\Omega_L)}\leq & \norm{ \chi_{\frac{L}{2}} w}_{H^1(\Omega^{\frac{L}{2}+1})}\leq C_{\Omega^{\frac{L}{2}+1}} \left(\norm{\chi_{\frac{L}{2}} w}_{L^2(\Omega^{\frac{L}{2}+1})}+\norm{\nabla\times (\chi_{\frac{L}{2}} w)}_{L^2(\Omega^{\frac{L}{2}+1})}\right) \\
			\leq & \ \widetilde{C}_{\Omega^{L}}(\norm{w}_{L^2(\Omega_L)}+\norm{\nabla\times w}_{L^2(\Omega_L)}),
		\end{aligned}
	\end{equation}
	for some constant $\widetilde{C}_{\Omega^L}>0$ depending on $L$.
	
	By contradiction, assume that $\Lambda_L(\Omega)=0$. Then, there exists a sequence $\{w_n\}_{n}\subset  {V}_L$ such that $\norm{w_n}_{L^2(\Omega_L)}=1,$ and $\norm{\nabla\times w_n}_{L^2(\Omega_L)}\leq \frac{1}{n}$. By \eqref{eq:concaappendixnorms}, this sequence is bounded in $H^1(\Omega_L)$ and, up to subsequences, it converges weakly to some $w\in H^1(\Omega_L)$ and, by compactness, strongly in $L^2(\Omega_L)$. Besides, $\nabla\times w=0$ by the lower semi-continuity of the norm with respect to weak continuity. Since a slight modification of \cite[Lemma A.2]{concaenglish} shows that the map $w\mapsto \norm{\nabla\times w}_{L^2(\Omega_L)}$ defines a norm in $V_L$, we conclude that $w\equiv 0$ in $\Omega_L$, which contradicts the original assumption $w\in V_L \setminus \{0\}$ and proves that $\Lambda_L(\Omega)>0$.
\end{proof}

 By slightly modifying the proof of Theorem \ref{theo:unrestuniquevortouter}, we obtain the next stability result:

\begin{theorem}\label{theo:unrestuniquevorinner}
	Fix $\alpha,\beta,\gamma\in\mathbb{R}$. Let  $	\Lambda_L(\Omega)>0$ and ${M}_{\rm out}(\alpha,\beta,\gamma)$ be as in \eqref{eq:Lambda_L} and \eqref{eq:M-vorticity}, respectively. 
    Assume that $(v,q)\in V_L(\Omega)\times L^2_{\rm loc}(\Omega)$ is such that $q$ is $L$-periodic with respect to $z$, and $(\U_{\alpha,\beta,\gamma}+v,\P_{\alpha,\gamma}+q)$ solves \eqref{ns-vorti}. If
	\begin{equation}\label{uniqueness-cond-vort-inner}
		{M}_{\rm out}(\alpha,\beta,\gamma)<\Lambda_L(\Omega),
	\end{equation}
	then $(v,q)\equiv (0,0)$ in $\Omega$.
\end{theorem}
\begin{proof}
	Since $(\U_{\alpha,\beta,\gamma},\P_{\alpha,\gamma})$ solves \eqref{ns-vorti}, the perturbation $(v,q)$ satisfies
	
	\begin{equation}\label{perturbNSinnervort}\begin{cases}
			-\Delta v + v \cdot \nabla v + {\U}_{\alpha,\beta,\gamma} \cdot \nabla v  + v \cdot \nabla  {\U}_{\alpha,\beta,\gamma} + \nabla q =0, \quad \nabla\cdot v =0 \quad \text{in} \quad \Omega, \\[5pt]
			v\cdot \nu =0,\quad
			(\nabla\times v)\times \nu=0 \quad
			\text{on}\quad \Gamma_2, \qquad \, v = 0\quad \text{on}\quad \Gamma_1\, ,
		\end{cases}
	\end{equation}and, as mentioned above, it is a classical solution to \eqref{perturbNSinnervort}.
	
	We denote $\Gamma_2\cap \overline{\Omega_L}$ by $\Gamma_{2,L}$.
	Taking the $L^2(\Omega_L)$-scalar product of \eqref{perturbNSinnervort} with $v$, integrating by parts and using the divergence-free and periodicity conditions yields
	
	\begin{equation}\label{dirnorm-vort-periodic}
		\begin{aligned}
			\int_{\Omega_L} |\nabla\times v|^2  =& \int_{\Gamma_{2,L}} ((\nabla\times v)\times \nu)\cdot v- \int_{\Omega_L}\left(v\cdot \nabla {\U}_{\alpha,\beta,\gamma}\right) \cdot v \\
			&+\int_{R_1^2<x^2+y^2<R_2^2} v_z(x,y,z) q(x,y,z)-v_z(x,y,z+L)q(x,y,z+L)\\[5pt]
			=&- \int_{\Omega_L}\left(v\cdot \nabla {\U}_{\alpha,\beta,\gamma}\right) \cdot v.
		\end{aligned}
	\end{equation}
	
	We then deduce from \eqref{dirnorm-vort-periodic} that

	\begin{equation*}
		\int_{\Omega_L} |\nabla\times v|^2 \leq {M}_{\rm out}(\alpha,\beta,\gamma)\int_{\Omega_L} |v|^2\leq \frac{{M}_{\rm out}(\alpha,\beta,\gamma)}{\Lambda_L(\Omega)}\int_{\Omega_L} |\nabla\times v|^2,
	\end{equation*}
	where
	\begin{equation*}
	\begin{aligned}
		{M}_{\rm out}(\alpha,\beta,\gamma)=
        \max_{\rho\in [R_1,R_2]}\left[\alpha^2 \frac{R_1^2R_2^4}{(R_2^2 -R_1^2)^2\rho^4} +\left(\frac{1}{\rho}
		\bigg(\frac{\beta-\gamma}{4}R_1^2
		-\frac{\beta}{8}\frac{R_2^2-R_1^2}{\log(R_2/R_1)}\bigg)
		+\rho\frac{\gamma}{4} \right)^2\right]^{\frac{1}{2}}.
	\end{aligned}
	\end{equation*}
	Hence, if \eqref{uniqueness-cond-vort-inner} holds, $\nabla\times v\equiv0$ in $\Omega_L$.  Therefore, we also deduce that $v\equiv0$ in $\Omega_L$ and, in turn,  that  $q\equiv0$ in $\Omega_L$.\end{proof}

\begin{remark}
	The pressure $\mathcal{P}_{\alpha,\gamma}$ is not periodic. However, Theorem \ref{theo:unrestuniquevorinner} establishes that no perturbations in $L^2_{\mathrm{loc}}(\Omega)$ of $\mathcal{P}_{\alpha,\gamma}$ that are $L$-periodic in $z$ are admissible. This periodicity constraint on the perturbation is necessary: without it, the perturbation $v = \mathcal{U}_{0,0,\gamma'}$, $q = \mathcal{P}_{0,\gamma'}$ would yield distinct solutions for arbitrary values of $\gamma'\in \mathbb{R}$.	
	This constraint has physical meaning: the parameter $\gamma$ relates to mass density. Thus, enforcing periodicity ensures that the perturbation preserves the mass density of the fluid.
\end{remark}

As a final remark, we emphasize that Theorem \ref{theo:unrestuniquevorinner} applies to any periodic perturbation. The velocity of the spiral Poiseuille-Couette flows is both helical and axially symmetric. Consequently, the constraint  \eqref{uniqueness-cond-vort-inner} implies that no other helical or axially symmetric solutions with vertical period $L$ exist. The question of whether \eqref{uniqueness-cond-vort-inner} can be improved within the helical or axially symmetric class remains open. On the one hand, since they are proper subspaces, one could expect that the upper bound in \eqref{uniqueness-cond-vort-inner} can be enlarged. On the other hand, the first bifurcation experimentally observed when increasing the magnitude of the boundary data enjoys axial symmetry, see \cite{Chossat}: this suggests that the stability thresholds might not be affected considering the smaller axially symmetric class.
\appendix
\section{Analysis of bounds for stability}

In this section we discuss the stability condition \eqref{small-cond-theo} as $R_1$ and $R_2$ vary.
Basically, we bound the Poincaré constant in a planar annulus. Similar bounds can also be obtained
for the other stability conditions \eqref{M-alphabeta-new2}, \eqref{uniqueness-cond-vort-outer}, \eqref{uniqueness-cond-vort-outer2} and \eqref{uniqueness-cond-vort-inner}.

\begin{proposition}\label{R1R2vary}
	Let $\alpha,\beta \in \mathbb{R}$ and assume that $(v, q)\in V_0(\Omega)\times L^2(\Omega)$ is such that $( \U_{\alpha,\beta}+v, \P_{\alpha,\beta}+q)$
	solves \eqref{SNSEcyl}-\eqref{nsstokes0}$_1$. Then, in order for the condition \eqref{small-cond-theo} to hold (which ensures that $(v,q)\equiv(0,0)$ in $\Omega$)
	the following behaviors are necessary:\smallskip 
    
    \begin{itemize}
	\item[$(i)$] if $R_2>0$ is given and $R_1\to0$, then the upper bounds for $|\alpha|$ and $|\beta|$ tend to $0$;
    
	\item [$(ii)$] if $R_1>0$ is given and $R_2\to\infty$, then the upper bounds for $|\alpha|$ and $|\beta|$ tend to $0$;
    
	\item [$(iii)$] if $R_1>0$ is given and $R_2\to R_1$, then the upper bounds for $|\alpha|$ and $|\beta|$ tend to infinity.
    \end{itemize}
\end{proposition}
\begin{proof} We start by proving that $\Lambda(\Omega)$ in \eqref{poincare-const} satisfies
	\begin{equation}\label{upperlambda}
		\Lambda(\Omega)\le\frac{10}{(R_2-R_1)^2}\, .
	\end{equation}
	For any $\eps>0$, we show the existence of a vector field $V_\eps\in H^1_0(\Omega)$ such that
	\neweq{Vepsilon}
	\frac{\|\nabla V_\eps\|_{L^2(\Omega)}^2}{\|V_\eps\|_{L^2(\Omega)}^2}=3\eps^3+\frac{10}{(R_2-R_1)^2}\, .
	\endeq
	We determine $V_\eps=(v_\eps,0,0)$ with two null components for some scalar $v_\eps\in H^1_0(\Omega)$, thus
	$$\|\nabla V_\eps\|_{L^2(\Omega)}=\|\nabla v_\eps\|_{L^2(\Omega)},\qquad\|V_\eps\|_{L^2(\Omega)}=\|v_\eps\|_{L^2(\Omega)}.$$
	Setting $v_\eps(x,y,z)=(1-\eps|z|)^+(R_2-\sqrt{x^2+y^2})(\sqrt{x^2+y^2}-R_1)$, with $f^+$ denoting the positive part of $f$, then we have
	$$
	\|v_\eps\|_{L^2(\Omega)}^2=4\pi\int_{0}^{1/\eps}(1-\eps z)^2dz
	\int_{R_1}^{R_2}\rho(R_2-\rho)^2(\rho-R_1)^2d\rho=\frac{\pi}{45\eps}(R_2+R_1)(R_2-R_1)^5\, .
	$$
	Moreover, after computing
	$$
	\nabla v_\eps (x,y,z)=\left(\begin{array}{ccc}
		(1-\eps z)\Big[\tfrac{(R_2+R_1)x}{\sqrt{x^2+y^2}}-2x\Big]\\
		(1-\eps z)\Big[\tfrac{(R_2+R_1)y}{\sqrt{x^2+y^2}}-2y\Big]\\
		\eps(\sqrt{x^2+y^2}-R_2)(\sqrt{x^2+y^2}-R_1)
	\end{array}\right)\qquad\forall z\in(0,1/\eps),
	$$which implies that
	\begin{align*}
		\|\nabla v_\eps\|^2_{L^2(\Omega)} &=
		4\pi\int_{0}^{1/\eps}(1-\eps z)^2dz\int_{R_1}^{R_2}\rho(R_2+R_1-2\rho)^2d\rho+4\pi\eps^2\int_{R_1}^{R_2}\rho(R_2-\rho)^2(\rho-R_1)^2d\rho\\
		&= \frac{2\pi}{9\eps}(R_2+R_1)(R_2-R_1)^3+\frac{\pi\eps^2}{15}(R_2+R_1)(R_2-R_1)^5\, ,
	\end{align*}
	we obtain \eq{Vepsilon}.
	From \eqref{boundPoincare} and \eqref{upperlambda} we have the following bounds:
	$$
	\max\left\{\frac{\pi^2}{2R_2^2},\frac{8}{(R_2^2 - R_1^2)\log (R_2/R_1)} \right\}\le\Lambda(\Omega)\le\frac{10}{(R_2-R_1)^2}\, ,
	$$
	from which we infer that
	\begin{eqnarray}
		\frac{\pi^2}{2R_2^2} &\le& \liminf_{R_1\to0}\Lambda(\Omega)\le\limsup_{R_1\to0}\Lambda(\Omega)\le\frac{10}{R_2^2}\qquad\forall R_2 >0\, , \label{asymp1}\\[5pt]
		4 &\le&\liminf_{R_2\to R_1}\big[(R_2-R_1)^2\Lambda(\Omega)\big]\le\limsup_{R_2\to R_1}\big[(R_2-R_1)^2\Lambda(\Omega)\big]\le10\quad\forall R_1>0 \,, \label{asymp3}\\[5pt]
		&& \lim_{R_2\to\infty}\Lambda(\Omega)=0\quad\forall R_1>0\, , \label{asymp2}
	\end{eqnarray}
	where in \eqref{asymp3} we have used that $\log(R_2/R_1)\leq R_2/R_1-1.$
	Next, let us rewrite $M$ in \eqref{M-alphabeta-new} as
	$$
	M(\alpha, \beta)= \left[\alpha^2\Phi_1(R_1,R_2)+\frac{\beta^2}{16}\Phi_2(R_1,R_2)\right]^{\tfrac{1}{2}},
	$$
	where
	$$
	\Phi_1(R_1,R_2):=\frac{R_2^4}{(R_2^2 -R_1^2)^2R_1^2}\, ,\quad\Phi_2(R_1,R_2):=\left(\frac{R_2^2-R_1^2}{R_1\log(R_2^2/R_1^2)}-R_1\right)^2\, ,
	\quad\mbox{for }0<R_1<R_2\, .
	$$
	Some calculus computations show that
	\begin{equation}\label{R1to0}
\lim_{R_1\to0}\Phi_1(R_1,R_2)=\lim_{R_1\to0}\Phi_2(R_1,R_2)=\infty\quad\forall R_2>0\, ,
	\end{equation}
	\begin{equation}\label{R2toinfty}
		\lim_{R_2\to\infty}\Phi_1(R_1,R_2)=\frac1{R_1^2}\, ,\quad\lim_{R_2\to\infty}\Phi_2(R_1,R_2)=\infty\quad\forall R_1>0\, ,
	\end{equation}
	\begin{equation}\label{R2toR1}
		\lim_{R_2\to R_1}\big[(R_2-R_1)^2\Phi_1(R_1,R_2)\big]=\frac14 \, ,\quad\lim_{R_2\to R_1}\Phi_2(R_1,R_2)=0\, \quad\forall R_1>0\, .
	\end{equation}
	
	To conclude the proof, we use \eqref{small-cond-theo} and the above bounds.
	Item $(i)$ is obtained by combining the upper bound in \eqref{asymp1} with \eqref{R1to0}.
	Item $(ii)$ is obtained by combining \eqref{asymp2} with \eqref{R2toinfty}.
	Item $(iii)$ is obtained by combining the lower bound in \eqref{asymp3} with \eqref{R2toR1}.
\end{proof}

\bigskip

\subsection*{Acknowledgements.} The research of E.B.\ and 
F.G.\ is supported by the grant \textit{Dipartimento di 
Eccellenza 2023-2027},
issued by the Ministry of University and Research (Italy); both are also partially supported by the Gruppo Nazionale
per l’Analisi Matematica, la Probabilità e le loro Applicazioni (GNAMPA) of the
Istituto Nazionale di Alta Matematica (INdAM). A.H.T. would like to thank the Department of Mathematics of Politecnico di Milano for the warm hospitality during the visit in which the discussions leading to this preprint were initiated.
\par\smallskip
\noindent
{\bf Data availability statement.} Data sharing not applicable to this article as no datasets were generated or analyzed during the current study.
\par\smallskip
\noindent
{\bf Conflict of interest statement}.  The authors declare that they have no conflict of interest.

\bigskip

\addcontentsline{toc}{section}{References}
\bibliographystyle{abbrv}
\bibliography{references}
\end{document}